%% file: main.tex
\documentclass[a4paper]{article}

\usepackage[utf8]{inputenc}

\usepackage[inline]{enumitem}

\usepackage{amssymb}
\usepackage{amsmath}
\usepackage{amsthm}
\usepackage{graphicx}
\usepackage[colorlinks=true, allcolors=blue]{hyperref}
\usepackage{mathtools}
\usepackage{nameref}
\usepackage{newunicodechar}
\usepackage{ragged2e}
\usepackage{scrextend}
\usepackage{caption}
\usepackage{subcaption}
\usepackage{enumitem}

\usepackage{graphicx}
\usepackage{wrapfig}
\graphicspath{ {./images/} }

\DeclarePairedDelimiter\abs{\lvert}{\rvert}%
\DeclarePairedDelimiter\ceil{\lceil}{\rceil}
\DeclarePairedDelimiter\floor{\lfloor}{\rfloor}

\newtheorem{theorem}{Theorem}[section]
\newtheorem{corollary}[theorem]{Corollary}

\newtheorem{lemma}[theorem]{Lemma}
\newtheorem{claim}[theorem]{Claim}
\theoremstyle{definition}
\newtheorem{remark}[theorem]{Remark}
\newtheorem{definition}[theorem]{Definition}

\newcommand{\lcm}{\mathrm{lcm}}

\setlist[enumerate]{nosep}

\begin{document}
\title{Random Van der Waerden Theorem}
\author{Ohad Zohar}
\date{}
\maketitle
\begin{abstract}
    \input{abstract}

\end{abstract}
\input{body}

\section*{Acknowledgements}
\input{acknowledgments}

\bibliographystyle{amsplain}
\bibliography{ref}

\end{document}

%% file: abstract.tex
    \noindent In this paper, we prove a sparse random analogue of the Van der Waerden Theorem. 
    We show that, for all $r > 2$ and all $q_1 \geq q_2 \geq \dotsb \geq q_r \geq 3 \in \mathbb{N}$, $n^{-\frac{q_2}{q_1(q_2-1)}}$ is a threshold for the following property:
    For every $r$-coloring of the $p$-random subset of $\{1,\dotsc,n\}$, there exists a monochromatic $q_i$-term arithmetic progression colored $i$, for some $i$.
    This extends the results of Rödl and Ruciński for the symmetric case $q_1 = q_2 = \dotsb = q_r$. 
    The proof of the 1-statement is based on the Hypergraph Container Method by Balogh, Morris and Samotij and Saxton and Thomason. 
    The proof of the 0-statement is an extension of Rödl and Ruciński's argument for the symmetric case.

%% file: body.tex
\section{Introduction}
For $n,q \in \mathbb{N}$ and $r \ge 2$  let $[n] \rightarrow (q)_r$ denote the property that, for every
$r$-coloring of $[n]=\{1,\dotsc,n\}$, there exists a monochromatic arithmetic progression of length~$q$. 
More generally, we denote by $[n] \rightarrow (q_1, \dotsc, q_r)$ the property that, for every
$r$-coloring of $[n]$, there exists $i$ such that there is some arithmetic progression of length $q_i$ colored $i$. 
A classical result in Ramsey theory due to Van der Waerden \cite{VDW}, states that, for every choice of
$q_1,\dotsc,q_r \in \mathbb{N}$, there exists $n_0$ such that for every $n \geq n_0$ we have $[n] \rightarrow (q_1, \dotsc q_r)$.
One might think of Van der Waerden's theorem as the arithmetic analogue to the graph-theoretic Ramsey Theorem. \par
This work aims to determine necessary and sufficient conditions for the property $[n] \rightarrow (q_1,\dotsc,q_r)$ to hold when we replace $[n]$ with a typical set of a given density.
We define $A_p$ as the random subset of a set $A$, where every element $a \in A$ belongs to $A_p$ with probability $p$, independently of all other elements of $A$.
Many interesting questions in Ramsey theory deal with determining the thresholds for values of $p$ for which it is no longer possible to color $A_p$ without introducing specific monochromatic substructures. \par
Even though the use of probabilistic methods in Ramsey theory has a long history, the study of Ramsey properties of random structures was initiated more recently by Frankl and Rödl~\cite{FR86}, who applied probabilistic methods to prove the existence of a graph $G$ with no $K_4$ for which every $2$-coloring must contain a monochromatic triangle. \par
In a series of papers, Rödl and Ruciński \cite{RR93, RR94, RR95}, determined the thresholds for the symmetric Van der Waerden property.
Let $[n]_p \rightarrow (q)_r$ denote the event that, for every $r$-coloring of $[n]_p$, there exists a monochromatic arithmetic progression of length $q$.
Similarly, let $[n]_p \rightarrow (q_1, \dotsc, q_r)$ denote the event that, for every
$r$-coloring of $[n]_p$, there exists a color $i$ such that there exists an arithmetic progression of length $q_i$ colored $i$. 
\begin{theorem} [Rödl and Ruciński] \label{thm:symmetric-random-van-der-waerden}
For $3 \leq q \in \mathbb{N}$ and every $r \ge 2$, there exist $c,C >0$ such that
\[ \lim_{n \to \infty} \mathbb{P}([n]_p \rightarrow (q)_r) = \begin{cases}
        1 & \text{if }p \geq C \cdot n^{-\frac{1}{q-1}}, \\
        0 & \text{if }p \leq c \cdot n^{-\frac{1}{q-1}}.
\end{cases}\]
\end{theorem} 

Similarly to the notation for the Van der Waerden property, one lets $G \rightarrow (F_1, \dotsc, F_r)$ denote the property that, for every $r$-coloring of $E(G)$, there exists $i$ such that there is a copy of $F_i$ colored $i$.
As well as Theorem \ref{thm:symmetric-random-van-der-waerden}, Rödl and Ruciński have also determined the threshold for the event $G(n,p) \rightarrow (F, \dotsc, F)$.
In 1997, Kohayakawa and Kreuter~\cite{KK97} initiated the study of the asymmetric case for graphs; they determined the threshold for the event $G(n,p) \rightarrow (C_1, \dotsc, C_r)$ where $C_1, \dotsc, C_r$ are all cycles, and conjectured the location of the threshold for general subgraphs $(F_1, \dotsc, F_r)$.
Several papers have since extended Kohayakawa and Kreuter's result to other families of subgraphs. 
For instance, Marciniszyn, Skokan, Sp\"{o}hel and Steger \cite{MarMarSkoSpo09}, showed that the conjecture holds when $F_1, \dotsc, F_r$ are all cliques.
More recently Mousset, Nenadov and Samotij~\cite{MouNenSam} proved an upper bound for the threshold function in the Kohayakawa--Kreuter conjecture for general subgraphs, extending a result of Gugelmann, Nenadov, Person, \v{S}kori\'{c}, Steger and Thomas \cite{GugNenRaj}, and settling the 1-statement.
Very recently, Liebnau, Mattos, Mendonça and Skokan~\cite{liebenau2020asymmetric}, have shown that the 0-statement holds for $r=2$ for any pair of cycles and cliques.
However, the 0-statement for general subgraphs remains open. \par
In this paper we prove the following natural analogue of the Kohayakawa--Kreuter conjecture for the Van der Waerden theorem:
\begin{theorem} \label{thm:random-van-der-waerden}
For every $r \ge 2$ and $q_1 \geq q_2 \geq \dotsb \geq q_r \in \mathbb{N}$, there exist $c,C >0$ such that
\[ \lim_{n \to \infty} \mathbb{P}([n]_p \rightarrow (q_1,\dotsc, q_r)) = \begin{cases}
1 & \text{if }p \geq C \cdot n^{-\frac{q_2}{q_1(q_2-1)}}, \\
0 & \text{if }p \leq c \cdot n^{-\frac{q_2}{q_1(q_2-1)}}.
\end{cases}\]
\end{theorem} 
Whether the threshold in Theorem \ref{thm:random-van-der-waerden} is sharp remains an interesting open question; so far it has only been shown that this is the case in $\mathbb{Z}_n$ when $r=2$ and $q_1=q_2$ by Friedgut, Han, Person and Schacht \cite{FriHanPerSch16}. \par
It is important to note that only the two largest lengths determine the threshold.
Therefore, in the proof of the 1-statement, it suffices to assume $q_2 = \dotsb = q_r$.
For the 0-statement, we will show that a proper coloring exists using only the first two colors, as is necessary for the case $r=2$, since that is also sufficient for all values of $r$.
By that reasoning, it is natural to divide the proof into the symmetric ($q_1 = q_2$) and asymmetric ($q_1 > q_2$) cases. \par
A generalization of Van der Waerden's theorem is the classical Rado theorem \cite{Rado}, which characterizes the so-called partition-regular matrices. We say a matrix $A$ is partition-regular if every finite coloring of the positive integers admits a monochromatic solution to the equation $Ax=0$.
Independently to this work, some progress has been made towards an asymmetric version of the random Rado theorem: given a sequence of partition-regular matrices $A_1, \dotsc, A_r$, for which values of $p$ does every $r$-coloring of $[n]_p$ admit a solution to $A_ix=0$ colored $i$? Aigner-Horev and Person \cite{AigPer19} obtained an upper bound for the threshold value, which implies the 1-statement in Theorem \ref{thm:random-van-der-waerden}. Soon after, Hancock and Treglown \cite{HanTre} obtained a matching lower bound for the case where every $A_i$ has rank one; this result implies the 0-statement in a special case of theorems \ref{thm:symmetric-random-van-der-waerden} and \ref{thm:random-van-der-waerden}, where $q_i=3$ for all $i$ . \par
The structure of this paper is as follows. 
In Section \ref{sec:preliminary}, we present several known results that are used in our proofs.
In Sections \ref{sec:symm-1} and \ref{sec:symm-0}, we present short proofs for the symmetric case, already proved by Rödl and Ruci\'{n}ski \cite{RR94}. We also correct an error in their proof for the 0-statement, which was independently discovered by Hancock and Treglown.
These sections are included here strictly for completeness; readers familiar with these results are encouraged to continue reading from Section \ref{sec:asymm-1} in which we prove the 1-statement for the asymmetric case.
Finally, in Section \ref{sec:asymm-0}, we complete the proof of Theorem \ref{thm:random-van-der-waerden} by proving the 0-statement for the asymmetric case.

\section{Preliminary results} \label{sec:preliminary}
In this section, we state several known results that are used throughout this paper.
The first result is the Hypergraph Container Lemma proved by Balogh, Morris and Samotij \cite{BalMorSam14}, and independently by Saxton and Thomason \cite{SaxTho}.
For an introduction to the various applications of this lemma, as well as the formulation used in this paper, we refer the reader to \cite{BalMorSam18}.
\begin{definition}
For a $k$-uniform hypergraph $H$ and a set $A \subset V(H)$ we define 
\[
    d(A) = \abs{\{e \in E(H) : A \subset e\}},
\] 
and for $\ell \in \{1, \dotsc, k\}$ we define 
\[
    \Delta_\ell(H) = \max\{d(A): A \subset V(H) \text{ and } \abs{A} = \ell \}.
\]
\end{definition}

\begin{definition}
    Let $H$ be a hypergraph, we denote the set of independent subsets of $V(H)$ by
    \[\mathcal{I}(H) = \{I \subset V(H) : \forall E \in E(H), E \not\subset I\}.\]
\end{definition}

\begin{theorem}[The Hypergraph Container Lemma] \label{thm:hypergraph-container} 
Let $k \in \mathbb{N}$ and $\epsilon > 0$. Let H be a nonempty $k$-uniform hypergraph, and suppose that:
\begin{center}
    $\Delta_\ell(H) \leq K \cdot \left(\frac{b}{v(H)}\right)^{\ell-1} \cdot \frac{e(H)}{v(H)}$
\end{center}
for some $b,K \in \mathbb{N}$ and every $\ell \in \{1,\dotsc , k\}$. 
Then, there exists a constant $D = D(\epsilon, k, K)$, a collection $\mathcal{C} \subset P(V(H))$ and a function $f\colon P(V(H)) \to \mathcal{C}$ such that: 
\begin{enumerate}
    \item for every $I \in \mathcal{I}(H)$, there exists $S \subset I$ with $\abs{S}{} \leq Db$ and $I \subset f(S)$, 
    \item each $C \in \mathcal{C}$ contains fewer than $\epsilon \cdot e(H)$ edges.
    \end{enumerate}
\end{theorem}

\begin{claim}\label{claim:degree-is-n}
    Let $H$ be the hypergraph encoding $q$-APs in $[n]$, where $V(H)=[n]$ and $E(H)$ is the set of arithmetic progressions of length $q$.
    Then, \[\Delta(H) = \Delta_1(H) \le n\].
\end{claim}
\begin{proof}
    We denote the number of arithmetic progressions of length $q$ in $[n]$ such that $k$ is the $i$th element by $d(k,i)$.
    One easily checks that $d(k,i)$ satisfies:
    \[
        d(k,i) = \begin{cases}
            \left\lfloor \frac{n-k}{q-1} \right\rfloor & \text{if } i = 1, \\
            \left\lfloor \frac{k-1}{q-1} \right\rfloor & \text{if } i = q, \\
            \min\{ \left\lfloor \frac{k-1}{i-1} \right\rfloor,  \left\lfloor \frac{n-k}{q-i} \right\rfloor \} & \text{otherwise}.
        \end{cases}
    \]
    We obtain that that the degree of $k$ in $H$ satisfies \[d(k) = \sum_{i=1}^q d(k,i) \le \frac{n}{q-1} + \sum_{i=2}^{q-1} d(k,i).\]
    Applying the bound  $\min\{ \left\lfloor \frac{k-1}{i-1} \right\rfloor,  \left\lfloor \frac{n-k}{q-i} \right\rfloor \} \le \frac{k-1 + n-k}{i-1 + q - i}< \frac{n}{q-1}$ we conclude that
    \[
        d(k) = \sum_{i=1}^q d(k,i) \le \frac{n}{q-1} + \sum_{i=2}^{q-1} d(k,i) \le \frac{n}{q-1} + (q-2) \cdot \frac{n}{q-1} \le n.
    \]
\end{proof}

\begin{remark} \label{rmk:hypergraph-container-q-APs}
    Let $H$ be the hypergraph encoding $q$-APs in $[n]$, where $V(H)=[n]$ and $E(H)$ is the set of arithmetic progressions of length $q$. Then: 
    \begin{enumerate}
        \item The number of edges of $H$ satisfies \[e(H)= \sum_{i=1}^{n-q+1} \left\lfloor \frac{n-i}{q-1} \right\rfloor =\Theta(n^2)\]
    since there are $\lfloor \frac{n-a}{q-1} \rfloor$ arithmetic progressions of length $q$ in $[n]$ with smallest element $a$. 
\item For $\ell = 1$, we have $\Delta_1(H) \leq  n$, by the previous claim.
\item For every $\ell \geq 2$, we have $\Delta_\ell(H) \leq q^2$ since choosing the indices of two elements determines the arithmetic progression. 
    \end{enumerate}
Therefore we may apply the Hypergraph Container Lemma with
    $b=qn^{\frac{q-2}{q-1}}$, since for $\ell=1$
\[
    \Delta_1(H) \leq n \leq K \cdot \frac{e(H)}{v(H)} 
\]
    and for $2 \leq \ell \le q$
\[
    \Delta_\ell(H) \leq q^2 \leq  q^{q-1} \leq K \cdot \frac{q^{q-1}n^{q-2}}{n^{q-1}} \cdot \frac{e(H)}{n} \leq K \cdot \left(\frac{qn^{\frac{q-2}{q-1}}}{v(H)}\right)^{\ell-1} \cdot \frac{e(H)}{v(H)},
\]
provided that $K$ and $n$ are sufficiently large.
\end{remark}
\begin{definition}
For brevity, we say a set of integers is $q$-AP-free if it contains no arithmetic progression of length $q$.
\end{definition}
We obtain the following container lemma for arithmetic progressions.

\begin{theorem} \label{thm:q-AP-container}
    For every integer $q \geq 3$ and $\epsilon > 0$, there exists a constant $D=D(\epsilon, q)$ such that for each $n \in \mathbb{N}$, there exists a collection $\mathcal{G} \subset P([n])$ and a function $f\colon P([n]) \to \mathcal{G}$
such that: 
\begin{enumerate}
    \item each $G \in \mathcal{G}$ contains fewer then $\epsilon n^2$ many $q$-APs, 
    \item for every $q$-AP-free subset $I \subset [n]$, there exists $S \subset I$ with $\abs{S}{} \leq D \cdot n^{\frac{q-2}{q-1}}$ and $I \subset f(S)$.
\end{enumerate}
\end{theorem}

We also require two classical results in probabilistic combinatorics, the first of which is Janson's inequality \cite{Janson}.
\begin{theorem}[Janson's inequality]
    Let $\Gamma$ be a finite set and let $\mathcal{S} \subset P(\Gamma)$. 
    For every $A \in \mathcal{S}$, let $I_A = 1$ if $A \subset \Gamma_p$ and $I_A = 0$ otherwise. \\
    Let $X = \sum_{A \in \mathcal{S}} I_A$ be the random variable counting the sets of $\mathcal{S}$ which are entirely contained in $\Gamma_p$.
    Set \[\mu = \mathbb{E}X \quad\text{ and }\quad \Delta = \sum_{\substack{(A,B) \in \mathcal{S}^2 \\ A\neq B, A \cap B \neq \emptyset}} \mathbb{E}[I_A \cdot I_B]. \]
Then, \[\mathbb{P}(X = 0) \leq e^{-\mu + \frac{\Delta}{2}}.\]
Moreover, if $\Delta \geq \mu$, then
\[\mathbb{P}(X = 0) \leq e^{-\frac{\mu^2}{2\Delta}.}\]
\end{theorem}

The second inequality is also known as The Extended Janson Inequality. For an introduction, as well as proofs for both inequalities, we refer the reader to \cite{AlonSpencerBook}. \par
We will also require a special case of Harris's inequality \cite{Harris}.

\begin{theorem}[Harris's inequality]
    Let $\Gamma$ be a finite set and let $f\colon P(\Gamma) \to \{0,1\}$ be an indicator function for some family of sets $\mathcal{A}$. We say $\mathcal{A}$ is increasing (equivalently decreasing), if $S_1 \subset S_2 \implies f(S_1) \leq f(S_2)$ (equivalently $f(S_1) \geq f(S_2)$ ).
    If $\mathcal{A}$ is increasing and $\mathcal{B}$ is decreasing, then
    \[ \mathbb{P}(\Gamma_p \in \mathcal{A} \cap \mathcal{B}) \leq \mathbb{P}(\Gamma_p \in \mathcal{A}) \cdot \mathbb{P}(\Gamma_p \in \mathcal{B}).\]
\end{theorem}

Again, for a more detailed discussion, we refer the reader to \cite{AlonSpencerBook}. \par
The final result in this section is a well-known quantitative version of Van der Waerden's Theorem due to Varnavides~\cite{Var59}.

\begin{lemma}\label{lemma:AP-supersaturation}
    For every $r \in \mathbb{N}$ and every $q \geq 2$, there exist $n_0 \in \mathbb{N}$ and $\epsilon > 0$ such that for all $n \geq n_0$, every $(r+1)$-coloring of $[n]$ contains at least $(r+1) \cdot \epsilon \cdot n^2$ monochromatic arithmetic progressions of length $q$. 
\end{lemma}
\begin{proof}
From Van der Waerden's theorem we have $W=W(r+1,q)$ such that every $(r+1)$-coloring of $[W]$ yields a monochromatic $q$-AP. Thus, in every coloring of $[n]$ every $W$-AP contains at least one monochromatic $q$-AP.
We observe: 
\begin{enumerate}[label=(\alph*)]
    \item $\# \text{$W$-APs in } [n]  = \Theta \left( \frac{n(n-W)}{W-1} \right)$. 
    \item Every $q$-AP is contained in at most $W^2$ many $W$-APs (fixing the indices of two terms in an arithmetic progression determines the progression). 
\end{enumerate}
We obtain from (a) $\Theta \left (\frac{1}{W} \cdot n^2\right )$ many $q$-APs, however they may be contained in multiple $W$-APs. Since by (b) every $q$-AP is counted at most $W^2$ times, 
the lemma follows with $\epsilon = \Theta \left( \frac{W ^ {-3}}{2(r+1)} \right)$.
\end{proof}

\section{The symmetric case 1-statement} \label{sec:symm-1}
In this section, we present a short proof to the following theorem (originally proved by Rödl and Ruci\'{n}ski \cite{RR95}):
\begin{theorem} \label{thm:symmetric-1-statement}
For every $r \in \mathbb{N}$ and $q \ge 3$, there exists $C>0$ such that the following holds: 
If $p \ge Cn^{-\frac{1}{q-1}}$ then a.a.s.\ every r-coloring of $[n]_p$ contains a monochromatic $q$-AP.
\end{theorem}

The general framework of the proof we present here is due to Nenadov and Steger \cite{NenSte16} who applied a similar argument in the setting of graphs.
We begin by describing a general outline of the proof.  \par
From The Hypergraph Container Lemma, we obtain a set of containers $\mathcal{G}$ for $q$-AP-free subsets of $[n]$.
Assume for contradiction that $[n]_p$ has a coloring with no monochromatic arithmetic progression of length $q$, and fix an arbitrary such coloring.
Let ${G_i} \in \mathcal{G}$ be the container for the $i$-th color class. Each container contains strictly fewer than $\epsilon n^2$ many $q$-term arithmetic progressions.
However, by Lemma \ref{lemma:AP-supersaturation} every coloring in $r+1$ colors must contain some color class with at least $\epsilon n^2$ arithmetic progressions.
In our case, we treat the remainder set $[n] \setminus \bigcup_i G_i$ as the final color class. Hence, it must have at least $\epsilon n^2$ arithmetic progressions and therefore has at least $\epsilon n$ many elements.
However, by definition, the remainder set (which depends only on $G_1,\dotsc,G_r$) and $[n]_p$ are disjoint.
Thus, the probability that $[n]_p$ has a coloring with no monochromatic $q$-term arithmetic progression that obeys the coloring constraints set by $(G_1, \dotsc, G_r)$ is at most $(1-p)^{\epsilon n}$.
Finally, we apply a union bound over all possible choices for the $r$-tuple of containers, obtaining that the probability that such $[n]_p$ has a proper $r$-coloring tends to zero.

\begin{proof} [Proof of Theorem \ref{thm:symmetric-1-statement}]
We say that a coloring is ``proper" if it contains no monochromatic $q$-AP. We wish to prove that a.a.s.\ $[n]_p$ admits no such coloring. 
Applying Theorem \ref{thm:q-AP-container} with $\epsilon = \epsilon(r)$ obtained from Lemma \ref{lemma:AP-supersaturation}, we obtain a family of containers $\mathcal{G}$ such that each $G \in \mathcal{G}$ contains fewer than $\epsilon n^2$ many $q$-APs.
If we suppose that there exists a proper coloring for $[n]_p$, then there are $q$-AP-free subsets $H_1,\dotsc,H_r$, such that $\bigcup\limits_{i=1}^r H_i = [n]_p$. \par
By Theorem \ref{thm:q-AP-container}, there exist a constant $D$ and a function $f\colon P([n]) \to \mathcal{G}$ such that for every $i$ there exists a set $S_i \subset H_i$, with $\abs{S_i}{} \leq D \cdot n^{\frac{q-2}{q-1}}$ and $H_i \subset f(S_i) = G_i$. 
By Lemma \ref{lemma:AP-supersaturation}, we obtain that for any coloring of $[n]$ with $r+1$ colors there must be a color class with at least $\epsilon n^2$ arithmetic progressions.  \par
Suppose we color the elements of $\bigcup G_i$ such that only elements of $G_i$ are colored $i$ and the elements of $[n] \setminus \bigcup\limits_{i=1}^r G_i$ are colored $r+1$. Since every $G_i$ contains fewer than $\epsilon n^2$ arithemetic progressions, the set of elements colored $r+1$ must contain $\epsilon n^2$ arithmetic progressions.
Since, by Claim \ref{claim:degree-is-n}, each $k \in [n]$ belongs to at most $n$ many $q$-APs we have 
    \[\abs{[n] \setminus \bigcup\limits_{i=1}^r G_i} \geq \frac{\epsilon n^2}{n} = \epsilon n.\]
\par In summary, the event ``There exists a proper coloring of $[n]_p$'' implies the following two events, for some $S_1, \dotsc, S_r \subset [n]$ with $\abs{S_i} \le Dn^{\frac{q-2}{q-1}}$ for every $i \in \{1,\dotsc, r\}$:
\begin{enumerate}[label=(\alph*)]
    \item $\bigcup_{i=1}^r S_i \subset [n]_p$,  
    \item $[n]_p \subset \bigcup_{i=1}^r f(S_i)$. 
\end{enumerate}
Note that (b) is equivalent to \[[n]_p \cap \left([n] \setminus \bigcup\limits_{i=1}^r f(S_i)\right) = \emptyset.\] Since $S_i \subset f(S_i)$ for every $i$,  (a) and (b) depend on disjoint subsets of $[n]$ and are independent events. 
Hence, the probability of both (a) and (b) occurring is: 
\begin{center}
    $\mathbb{P}\left(S_1,..., S_r \subset [n]_p \land [n]_p \subset \bigcup\limits_{i=1}^r f(S_i)\right) \leq p^{\abs{\bigcup\limits_{i=1}^r S_i}{}} \cdot (1-p)^{\epsilon n}.$ 
\end{center}
Givan a set $S$ with $\abs{S}=s$, there are at most $2^{rs}$ sequences $(S_1,\dotsc,S_r)$ such that $S=\bigcup S_i$.
Taking the union bound over all choices of $S_1,...,S_r$ (grouping by ${s=\abs{\bigcup\limits_{i=1}^r S_i}{}}$) we obtain 
\begin{align*}
    \mathbb{P}([n]_p \ \text{admits a proper coloring}) &\leq \sum_{(S_1,\dotsc,S_r)} p^{\abs{\bigcup\limits_{i=1}^r S_i}{}} \cdot (1-p)^{\epsilon n} \\
    \leq (1-p)^{\epsilon n} \sum_{s=1}^{Drn^{\frac{q-2}{q-1}}}\binom{n}{s}2^{rs}p^s 
    &\leq e^{-\epsilon pn}\sum_{s=1}^{Drn^{\frac{q-2}{q-1}}}\left(\frac{en2^rp}{s}\right)^s .
\end{align*}
\par Recall that $p = Cn^{-\frac{1}{q-1}}$. Since $x \mapsto \left(\frac{ea}{x}\right)^x$ is increasing for $x \leq a$, by choosing $C$ sufficiently large we obtain that $Drn^{\frac{q-2}{q-1}} \le \delta pn$ for some $\delta = \delta(C,D,r) > 0$ which can be made arbitrarily small by choosing $C$ sufficiently large. Hence,
    \begin{align*}
    \sum_{s=1}^{Drn^{\frac{q-2}{q-1}}}\left(\frac{en2^rp}{s}\right)^s
    &\leq Drn^{\frac{q-2}{q-1}} \cdot \left(\frac{en2^rp}{Drn^{\frac{q-2}{q-1}}}\right)^{Drn^{\frac{q-2}{q-1}}} \\
    &\leq \delta pn \cdot \left(\frac{e2^r}{\delta}\right )^{\delta pn}
    \leq e^{\frac{\epsilon pn}{2}}
\end{align*}
for sufficiently large $C$, since $(\frac{1}{\delta})^\delta \to 1$ as $\delta \to 0$.
Hence: 
\[\mathbb{P}([n]_p \ \text{admits a proper coloring}) \leq e^{\frac{-\epsilon p n}{2}} \to 0. \qedhere \] 
\end{proof}

\section{The symmetric case 0-statement} \label{sec:symm-0}
In this section, we prove the following theorem (originally proved by Rödl and Ruci\'{n}ski \cite{RR95}):
\begin{theorem} \label{thm:symmetric-0-statement}
For any integer $q \geq 3$, there exists $c>0$ such that for $p=c \cdot n^{-\frac{1}{q-1}}$, $[n]_p$ can a.a.s.\ be colored by two colors with no monochromatic $q$-term arithmetic progression.
\end{theorem}
\noindent Note that this implies the 0-statement in Theorem \ref{thm:symmetric-random-van-der-waerden} for any $r \geq 2$.
The proof here is a specialization of the proof by Rödl and Ruci\'{n}ski \cite{RR97} to the random Rado partition theorem, and is included here both for the sake of completeness and as an introduction to the techniques used for the asymmetric case in Section \ref{sec:asymm-0}. The proof consists of two main lemmas.
The deterministic lemma states that every non-2-colorable uniform hypergraph must contain one of a small family of hypergraphs which we refer to as $2$-blocking hypergraphs (which will be defined next).
Then, the probabilistic lemma states that in the random hypergraph of arithmetic progressions in $[n]_p$, the subhypergraphs mentioned in the deterministic lemma almost surely do not appear.
We begin by defining several families of hypergraphs.
\begin{definition}
    A \emph{simple path} is a hypergraph consisting of edges $E_1,\dotsc,E_\ell$, for $\ell \geq 1$, such that $$\abs{E_i \cap E_j} =  \begin{cases}
    1 & \text{if }\abs{i-j}=1, \\
0 &  \text{if }\text{otherwise.}

\end{cases}
$$
A \emph{fairly simple cycle} is a hypergraph that consists of a simple path $(E_1,\dotsc,E_\ell)$, with $\ell \geq 2$, and an edge $E_0$ such that
$$\abs{E_0 \cap E_i} = \begin{cases}
1 & \text{if } i = 1, \\
0 & \text{if } i = 2, \dotsc , \ell - 1, \\
s & \text{if } i = \ell,
\end{cases}
$$
for some $s \geq 1$, and such that $E_0 \cap E_1 \cap E_\ell = \emptyset$. A fairly simple cycle is said to be \emph{simple} if $s = 1$; otherwise, we say it is \emph{special}. 
A path $P$ in a hypergraph $H$ is said to be \emph{spoiled} if it is not an induced subhypergraph of $H$. We call an edge $E \subset V(P)$ such that $E \notin E(P)$ a \emph{spoiling edge} for $P$. \par
The length of a path or a fairly simple cycle is the number of edges in it.
A subhypergraph $H_0$ of $H$ is said to have a \emph{handle} if there is an edge $E$ in H such that $\abs{E} > \abs{E \cap V(H_0)} \geq 2$.
\end{definition}

\begin{definition}
    We call a $q$-uniform hypergraph \emph{$2$-blocking} if it is one of the following:
    \begin{enumerate}
        \item A special cycle.
        \item A simple cycle with a handle.
        \item A spoiled path.
        \item For $q=3$, the 3-uniform, 2-regular, 6-vertex, simple hypergraph, which we call ``the reduced Fano plane'' (See Figure \ref{fig:reduced-fano}) \footnote{The name stems from the fact that this hypergraph is exactly the Fano plane with one vertex removed. It is also sometimes refered to as a Pasch configuration.}.
    \end{enumerate}
\end{definition}

\begin{lemma}[The determinisic lemma]\label{lemma:deterministic-2-coloring}
Let $q \geq 3$ be some integer, and let $H$ be a q-uniform hypergraph which is not 2-colorable.
Then $H$ contains a $2$-blocking hypergraph.
\end{lemma}
\begin{proof}
Recall that we say a hypergraph is 3-edge-critical if it cannot be properly colored with two colors, but any proper subhypergraph of it is 2-colorable. 
We may assume that $H$ is 3-edge-critical; otherwise, we may replace it with one of its 3-edge-critical subhypergraphs.  
\begin{claim}
If $H$ is a 3-edge-critical hypergraph, then for every edge $E \in H$ and for every vertex $v \in E$ there is an edge $E' \in H$ such that $E \cap E' = \{v\}$.
\end{claim}
\begin{proof}
    Let $H$ be 3-edge-critical, and suppose that there are an edge $E \in H$ and a vertex $v \in E$ such that no other edge intersects $E$ in exactly $\{v\}$; thus, every edge $E'$ that contains $v$ also contains another vertex of $E$.
By the 3-edge-criticality, $H$ can be colored red-blue in a way such that only $E$ is monochoromatic, say it is blue. Now, by changing the color of $v$ to red, $E$ is no longer monochromatic, and neither is any other edge that contains $v$, contradicting the fact that $H$ is not 2-colorable.
\end{proof}
Let $P = (E_1,\dotsc,E_\ell)$ be a longest simple path in $H$. By the claim, $\ell \geq 2$. Let $x,y$ be two of the vertices that belong only to $E_1$, and let $E_x, E_y$ be two edges of $H$ that intersect $E_1$ only in $x$ and $y$ respectively. 
Since $P$ is maximal, $h_z = \abs{E_z \cap V(P)} \geq 2$, for $z = x,y$. \par
We may also assume that $E_x, E_y$ only intersect one another and each edge of $P$ in at most a single vertex (otherwise we obtain a special cycle).
Let $i_z = \min \{i \geq 2 : E_z \cap E_i \neq \emptyset \}$, and without loss of generality assume $i_y \leq i_x$.
If $h_z = q$ for some $z$ then $P$ is a spoiled simple path.
Otherwise, either the edges $E_1, \dotsc , E_{i_y}, \dotsc, E_{i_x}, E_x$ form a fairly simple cycle, to which $E_y$ is a handle (see Figure \ref{fig:deterministic-lemma-cycle}) or $E_y \subset (E_1 \cup \dotsb \cup E_{i_x} \cup E_x)$. \par
Assuming the latter case, since $h_y < q$ and we assumed that $\abs{E_x \cap E_y} \le 1$, we must have $E_x\cap E_y = \{u\}$ for some vertex $u \notin V(P)$. We now split into several cases. 
First, if $i_y < i_x$ then the edges $E_1, \dotsc ,E_{i_y}, E_y$ form a fairly simple cycle to which $E_x$ is a handle (since it only intersects the cycle at $x$ and $u$).
Finally, if $i_y = i_x = i$ then either $E_x, E_y, E_i$ form a fairly simple cycle to which $E_1$ is a handle, or $i=2, q=3$ and $E_1, E_2, E_x, E_y$ form the reduced Fano plane.
\end{proof}
\begin{figure}[h]
\centering
\begin{subfigure}[b]{.45\textwidth}
\centering
\includegraphics[width=\textwidth]{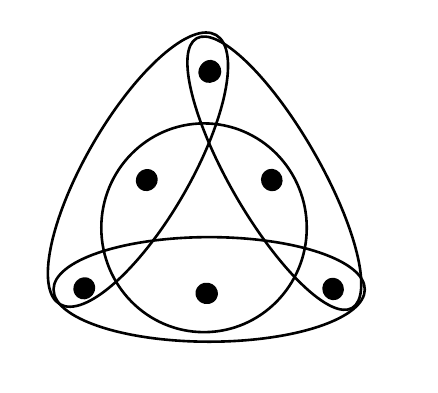}
\caption{The reduced Fano plane}
\label{fig:reduced-fano}
\end{subfigure}
\begin{subfigure}[b]{.45\textwidth}
\centering
\includegraphics[width=\textwidth]{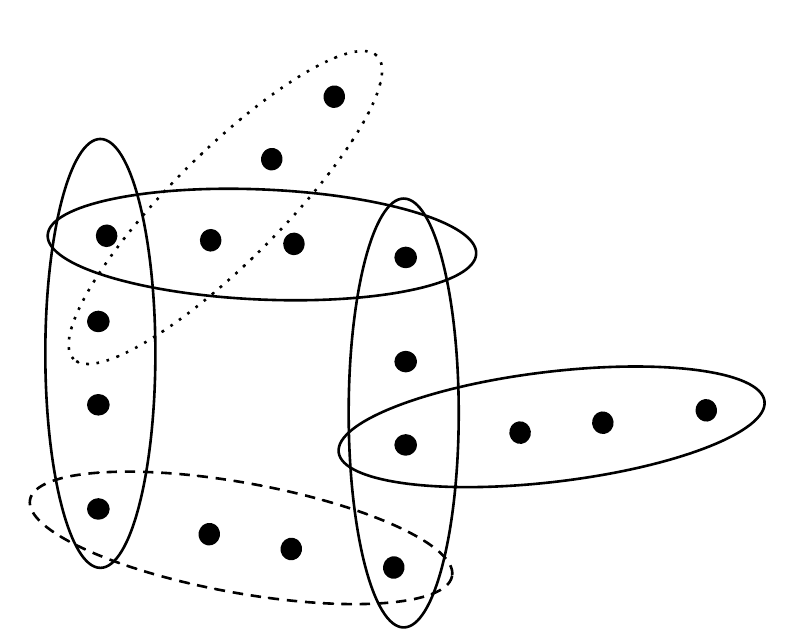}
\caption{$E_x$ (dashed) forms a cycle with $P$ to which $E_y$ (dotted) is a handle.}
\label{fig:deterministic-lemma-cycle}
\end{subfigure}
\end{figure}

\begin{lemma}[The probabilistic lemma]
\label{lemma:probabilistic-lemma}
Let $H$ be the hypergraph with vertex set $V(H) = [n]$, whose edge set is the set of $q$-APs, and let $H_p$ be its random subhypergraph induced by $[n]_p$. If $p=c \cdot n^{-\frac{1}{q-1}}$, then a.a.s.\ $H_p$ contains no $2$-blocking hypergraph, provided that $c$ is sufficiently small.
\end{lemma}
\begin{proof}
We apply a first-moment argument to several random variables.
Let $B$ be a large enough constant, we will show that a.a.s.\ no path of length $B\log{n}$ exists in $H_p$. We will then show that a.a.s.\ no 2-blocking graph with fewer than $B\log{n}$ edges exist. \par
The following calculations are used several times throughout the arguments.
First, the number of choices for a $q$-term arithmetic progression in $[n]$ is $O(n^2)$, and the probability that all of its $q$ elements belong to $[n]_p$ is $(cn^{-\frac{1}{q-1}})^q = O(n^{-1}p)$, so the expected number of $q$-term arithmetic progressions in $[n]_p$ is $O(np)$. \par
We say that an edge $A$ extends a simple path $P=(E_1, \dotsc, E_\ell)$ in $H$, if $A \cap V(P) = \{v\}$ and $\{i : v \in E_i\}$ is exactly $\{1\}$ or $\{\ell\}$.
Observe that the number of choices for an arithmetic progression $A$ that extends a simple path $P$ is bounded by $2 \cdot (q - 1) \cdot q \cdot n$, and that the probability that $ A \setminus V(P) \subset [n]_p$ is $ (cn^{-\frac{1}{q-1}})^{q-1}$. Therefore, we may bound the expected number of possible edge choices for extending a simple path in $H_p$ by some constant $c_r$, which may depend on $q$ but not on $n$, and may be made arbitrarily small by changing $c$. \par 
The number of arithmetic progressions containing a fixed set $\{v_1, \dotsc, v_s\}$ with $s \geq 2$, is bounded by a constant (for instance $\binom{q}{s}$ is a trivial bound). Hence the expected number of arithmetic progressions in $H_p$ that contain $\{v_1, \dotsc, v_s\}$ is $O(p^{q-s})$, assuming $\{v_1, \dotsc, v_s\} \subset [n]_p$. \par
Since the number of simple paths of length $t$ in $H$ is at most $O(n^{1+t})$, 
and a simple path of length $t$ must contain $q$ vertices for the first edge and $q-1$ vertices for every subsequent edge,
the expected number of simple paths of length $t$ in $H_p$ is 
\[O(n^{1+t} \cdot p^{q + (t-1)(q-1)}) = O(np \cdot c_r^t).\] 
We are now ready to proceed with the argument. \par
Let $U$ be the random variable counting the number of simple paths in $H_p$ of length at least $B \cdot \log{n}$. We bound the expected value of $U$ by summing over different lengths of paths. 
By the above computations,
$$\mathbb{E}U \leq O(\sum_{t \geq B\log n} np \cdot c_r ^ t) = o(1),$$
provided that $c_r$ is sufficiently small and $B$ is sufficiently large.

Let $W$ be the random variable counting the number of special cycles in $H_p$.
For a given edge there are only a constant number of edges that interesect it in more than one vertex, therefore the number of such edge pairs in $H$ is $O(n^2)$.
We may bound the number of special cycles of length $t$ in $H$ by first fixing $E_0$ and $E_{t-1}$, then fixing a simple path of length $t-3$ starting from some vertex of $E_0$ and finally choosing $E_{t-2}$. Note that we have $O(1)$ many choices for $E_{t-2}$ since it must include exactly one vertex from $E_{t-3}$ and one vertex from $E_{t-1}$. Moreover, these vertices cannot coincide as $E_0 \cap E_1 \cap E_{t-1} = \emptyset$. In total we obtain that there are $O(n^2 \cdot n^{t-3}) = O(n^{t-1})$ such cycles in $H$. 
Since we must have at least $q+1$ vertices for $E_0$ and $E_{t-1}$, an additional $(t-3)\cdot (q-1)$ vertices for $E_1, \cdots, E_{t-3}$, and exactly $q-2$ additional vertices for $E_{t-2}$ we require a total of $2q-1 + (t-3)\cdot (q-1) = 1 + (t-1)\cdot (q-1)$ vertices in $H_p$.
Hence,
$$\mathbb{E}W = O(\sum_{t>2}n^{t-1}p^{1 + (t-1) \cdot (q-1)}) = O(p \sum_{t>2}c_r^{t-1})=O(p) = o(1),$$
provided that $c_r$ is sufficiently small. \par
Let $X$ be the random variable counting the number of simple cycles with handles of length at most $B\log n$ in $H_p$.
We denote the length of the cycle by $t$, and the size of the intersection between the handle and the cycle by $k$.
Similarly to the previous argument, there are $O(n^t)$ many cycles of length $t$ in $H$. 
Since the handle must attach in at least two vertices of the cycle we may bound the number of handles in $H$ by $O(\log^2{n})$.
For fixed values of $t$ and $k$, such a configuration requires $t \cdot (q-1)$ vertices for the cycle and an additional $q-k$ vertices for the handle.
Summing over $t$ and $k$,
$$\mathbb{E}X = O(\sum_{t=3}^{B\log n}\sum_{2 \leq k \leq q-1} n^t \cdot p^{t\cdot (q-1) + q-k} \cdot \log^2{n}) = O(p\log^3{n})= o(1).$$
\par Let $Y$ be the random variable counting the number of spoiled simple paths of length less than $B\log n$ in $H_p$.
Let $E$ be an induced edge that is not one of the edges of a spoiled path. We now split into two cases.
First, suppose that $E$ intersects some edge in at least two vertices, we obtain a path $(E_1, \dotsc, E_\ell)$ and an edge $E_0 = E$, such that $\abs{E_0 \cap E_1} = s \geq 2,\  \abs{E_0 \cap (E_\ell \setminus E_1)} = t \geq 1$.
If $t = 1$ and $E_0 \cap E_1 \cap E_\ell = \emptyset$  this yields a special cycle, and thus the expected number of such hypergraphs in $H_p$ is $o(1)$.
Otherwise, $t' =\abs{E_\ell \cap (E_0 \cup E_1)} \ge 2$. As $\abs{E_1 \cap E_\ell} \leq 1$ we have $s+t' \leq q+1$. Hence, the expected number of choices for $E_0, E_1, E_\ell$ is 
\[O(\sum_{s,t'}np \cdot p^{q-s} \cdot p^{q-t'}) = O(\sum_{s,t'}np^{2q - (s + t') + 1}) = O(np^q) = o(1). \] \par
Assuming that $E$ intersects every edge of a path $(E_1, \dotsc, E_\ell)$ in at most one vertex, we may define an ordering function $f\colon E \to [\ell]$ by $f(v) = \min\{i: v \in E_i \}$ and order the vertices of $E$ by the values of $f$.
Observe the shortest sub-path containing the first three vertices $ \{v_1,v_2,v_3\} \subset E$, and denote $t_1 = f(v_2) - f(v_1)$ and $t_2 = f(v_3) - f(v_2)$. 
Let $Y'$ count such configurations in H. 
We have an expected $O(np)$ many choices for $E$, and $O(c_r^{t-1}p^{q-2})$ many choices for a path of length $t$ between two fixed vertices.
Summing over $t_1$ and $t_2$ we obtain
$$\mathbb{E}Y' = O(\sum_{t_1\geq 1}\sum_{t_2\geq 1} np \cdot c_r^{t_1-1}p^{q-2} \cdot c_r^{t_2-1} p^{q-2} ) = O(np^{2q-3}) = O(p^{q-2})=o(1).$$
Thus,
\[ \mathbb{E}Y = o(1). \]
\par Finally, let $Z$ be the random variable counting the number of copies of the reduced Fano plane in $H_p$.
First, we show that there are $O(n^2)$ copies of the reduced Fano plane in $H$.
Suppose that  $\{x_1, x_2, x_3\}$ is an edge in $H$, since there are $O(n^2)$ choices for an arithmetic progression of length three it suffices to show that there are at most $O(1)$ many choices for $x_4,x_5,x_6$ such that $\{x_1, \dotsc, x_6\}$ induce a copy of the reduced Fano plane.
Denote $\vec{v} = (x_1, x_2, x_3)$ and $\vec{x} = (x_4,x_5,x_6)$.
Since every pair of vertices in $\{x_4, x_5, x_6\}$ form a $3$-AP with one vertex in $\{x_1,x_2,x_3\}$, we obtain the following family of equations:
\[
    A\cdot \vec{x} = \vec{v} \quad\text{ for }\quad A = \begin{pmatrix} a&b& 0 \\ 0& c& d \\ e& 0& f \end{pmatrix} \quad\text{ with }\quad a,b,c,d,e,f \in \{-1,\frac{1}{2}, 2\}.
\] \par
Since $\det(A) = acf + bde \ne 0$ for all choices of $A$ we obtain that $A$ is invertible, and therefore fixing $\vec{v}$ determines $\vec{x}$.
Since the reduced Fano plane contains six vetices we obtain that the expected number of such configurations is $O(n^2p^6) = O(n^{2 - \frac{6}{2}}) = O(\frac{1}{n}) = o(1)$. \par
Thus, by Markov's inequality $\mathbb{P}(U=W=X=Y=Z=0) \to 1$ as $n \to \infty$ completing the proof.
\end{proof}
With the two lemmas in hand, the proof of Theorem \ref{thm:symmetric-0-statement} is immediate.
\begin{proof} [Proof of Theorem \ref{thm:symmetric-0-statement}]
Let $H$ be the hypergraph with vertex set $V(H)=[n]_p$, whose edge set is the set of $q$-term arithmetic progressions. 
By Lemma \ref{lemma:probabilistic-lemma}, $H$ a.a.s.\ contains no $2$-blocking hypergraph.
Therefore, by Lemma \ref{lemma:deterministic-2-coloring}, $H$ is 2-colorable. 
\end{proof}

\section{The asymmetric case 1-statement} \label{sec:asymm-1}
In this section, we prove the following result:
\begin{theorem}\label{thm:asymmetric-1-statement}
For every $r \geq 2$ and  $q_1 > q_2 \geq \dotsb \geq q_r \in \mathbb{N}, q_r \geq 3$, there exists $C>0$ such that the following holds: \\
If $p \ge Cn^{-\frac{q_2}{q_1(q_2-1)}}$ then a.a.s.\ for every r-coloring of $[n]_p$ there is some $i$ such that there exists a monochromatic $q_i$-AP colored $i$.
\end{theorem}
We begin by describing a rough outline of the proof, which follows similar ideas to \cite{GugNenRaj}. 
We say that a coloring is proper if it contains no monochromatic $q_i$-AP colored $i$; we wish to prove that a.a.s.\ $[n]_p$ admits no such coloring. 
First, we reduce the problem by showing that every proper coloring may be modified to yield a ``good" coloring where, additionally, every element not in a $q_1$-AP is colored 1.
Thus, it would suffice to show that no such good coloring exists.\par
Using \nameref{thm:hypergraph-container} we obtain a set of containers $\mathcal{G}$ for the $q_2$-AP-free subsets of $[n]$.
Supposing for contradiction that a good coloring exists, we fix one arbitrary such coloring along with containers $G_2, \dotsc, G_r \in \mathcal{G}$ for all the color classes but the first. \par
For $i \ge 2$, we denote the set of elements colored $i$ by $I_i$.
We denote the remainder set $I_1 = [n]_p \setminus (I_2 \cup \dotsb \cup I_r)$.
Since, by Lemma~\ref{lemma:AP-supersaturation}, $[n] \setminus (G_2 \cup \dotsb \cup G_r)$ contains at least $\epsilon n^2$ many arithmetic progressions of length $q_1$, the set of elements colored $1$ is unlikely to be $q_1$-AP-free.
Indeed, in Lemma \ref{lemma:no-q-ap-janson-bound}, we show that the probability that the set of elements colored $1$ contains no $q_1$-AP is exponentially small in $n^2p^{q_1}$. \par
Note, that the existence of a good coloring implies two events: First, the signature sets for the containers $G_2, \dotsc, G_r$ must all be covered by $q_1$-APs in $[n]_p$, and second, the remainder set $I_1$ must not contain a single $q_1$-AP.
Using Harris's inequality, we show that these two events are negatively correlated. \par
Finally, we apply a union bound over all possible choices for the tuple of containers $(G_2, \dotsc, G_r)$, by iterating over their signature sets $S_2,\dotsc, S_r$. 
A critical part of the union bound argument is Lemma  \ref{claim:APs-are-mostly-indpendent}, which roughly states that if $p=\Theta(n^{-\frac{q_2}{q_1(q_2-1)}})$ then typically $q_1$-APs rarely intersect, and therefore most of $S = \bigcup S_i$ is covered by isolated $q_1$-APs in $[n]_p$. 
\par
We begin by proving the following lemmas:

\begin{lemma}
\label{lemma:no-q-ap-janson-bound}
Suppose $\mathcal{A}$ is a collection of $\Omega(n^2)$ $q$-APs in $[n]$, and $np^{q-1} \ll 1$. Then,
$$ \mathbb{P}([n]_p \text{ does not contain any member of } \mathcal{A}) \leq \exp(-\Omega(n^2p^q)).$$
\end{lemma}
\begin{proof}
Enumerate the elements of $\mathcal{A} = \{E_i : i \in I\}$. For each $i \in I$, let $X_i$ be the indicator random variable for the event $E_i \subset [n]_p$ and let $X = \sum X_i$. Observe that
\begin{align*}
    \mu &= \mathbb{E}X = \Omega(n^2p^q), \\
    \Delta &= \sum_{\substack{i \neq j \\ E_i \cap E_j \neq \emptyset}}\mathbb{P}(E_i \cup E_j \subset [n]_p) = \sum_{i \in I}\sum_{1\leq k \le q-1}\sum_{\substack{j \in I \\ \abs{E_i\cap E_j} = k}} p^{2q-k}.
\end{align*}
Note that for a fixed $i$ and $k>1$ there are only at most $q^2 = O(1)$ many $j$s such that $\abs{E_i \cap E_j} = k$, and for $k=1$ there at most $O(n)$ such $j$s.
For a fixed $i$ this implies,
\[\sum_{1\leq k \leq q-1}\sum_{\substack{j \in I\\ \abs{E_i\cap E_j} = k}} p^{2q-k} = O(np^{2q-1} + \sum_{2\leq k \leq q-1} p^{2q-k}) = O(np^{2q-1} + p^{q+1}),\]
and thus, as $\abs{\mathcal{A}} = O(n^2)$,
\[\Delta = O(n^3p^{2q-1} + n^2p^{q+1})\ .\]
Moreover, by our assumption that $np^{q-1} \ll 1$ we have \[n^3p^{2q-1} +  n^2p^{q+1}\ll n^2p^q. \] 
Hence, by Janson's inequality, \[ \mathbb{P}(X=0) \leq \exp(-\mu + \frac{\Delta}{2}) = \exp(-\Omega(\mu)) = \exp(-\Omega(n^2p^q))\]
concluding the proof.
\end{proof}

\begin{definition}\label{def:mostly-independent}
    We say that a $q_1$-AP in $[n]_p$ is \emph{isolated}, if it does not intersect any other $q_1$-AP in $[n]_p$.
We define the following random variables:
\begin{align*}
    Q &= \bigcup\{A \subset [n]_p : A \text{ is a $q_1$-AP}\}, \\
    Q_I &= \bigcup\{A \subset [n]_p : A \text{ is an isolated $q_1$-AP}\}.
\end{align*}
Let $\delta = \frac{\min\{q_2, q_1-q_2\}}{2q_1(q_2-1)}$.
If $\abs{Q\setminus Q_I} < n^{1-\frac{1}{q_2-1} - \delta}$ we say that $Q$ is \emph{mostly independent}.
\end{definition}
\begin{lemma}\label{claim:APs-are-mostly-indpendent}
    If $p=O(n^{-\frac{q_2}{q_1(q_2-1)}})$ then $Q$ is mostly independent a.a.s.
\end{lemma}
\begin{proof}
First we compute the expected number of sets that are a union of two intersecting $q_1$-APs. 
First, since $\frac{1}{q_2 - 1} - \frac{q_2}{q_1(q_2-1)} = \frac{q_1 - q_2}{q_1(q_2-1)} \ge 2\delta$,
\begin{align*}
    \mathbb{E}(\text{\#pairs of $q_1$-APs sharing exactly one element}) &= O(n^2 \cdot n \cdot p^{2q_1 - 1}) \\
    = O(n^{1 - \frac{2}{q_2 - 1} + \frac{q_2}{q_1(q_2 - 1)}})  &= O(n^{1 - \frac{1}{q_2-1} - 2\delta}).
\end{align*}
Second, for every $1 < m < q_1$ 
\begin{align*}
    \mathbb{E}(\text{\#pairs of $q_1$-APs sharing exactly $m$ elements}) &= O(n^2 \cdot p^{2q_1 - m}) \\
                                                                 &= O(n^{-\frac{2}{q_2-1} + \frac{mq_2}{q_1(q_2-1)}}).
\end{align*}
Since $\frac{mq_2}{q_1(q_2-1)} \le \frac{(q_1-1)q_2}{q_1(q_2-1)} = 1 + \frac{1}{q_2-1} - \frac{q_2}{q_1(q_2-1)} \le 1 + \frac{1}{q_2-1} - 2\delta$,
\begin{align*}
    \mathbb{E}(\text{\#pairs of $q_1$-APs sharing $m$ elements}) 
                                                                 &= O(n^{1 - \frac{1}{q_2 - 1} - 2\delta}).
\end{align*}
In particular, since $\abs{Q\setminus Q_I} \le \abs{\{(A,B): A \cap B \ne \emptyset, A,B \in Q\}}$,
\[\mathbb{E}\abs{Q \setminus Q_I} \le O(n^{1 - \frac{1}{q_2 - 1} - 2\delta}).\]
Finally, we obtain from Markov's inequality,
\begin{align*}
    \mathbb{P}(\abs{Q \setminus Q_I} \ge  n^{1 - \frac{1}{q_2 - 1} - \delta}) \le 
O\left(\frac{n^{1 - \frac{1}{q_2 - 1} - 2\delta}}{n^{1 - \frac{1}{q_2 - 1} - \delta}}\right) =
O(n^{-\delta})  = o(1).
\end{align*}
Hence, a.a.s.\ 
\[\abs{Q \setminus Q_I} <n^{1-\frac{1}{q_2-1} - \delta}. \qedhere\] 
\end{proof}
We are now ready to prove Theorem \ref{thm:asymmetric-1-statement}.
\begin{proof}[Proof of Theorem \ref{thm:asymmetric-1-statement}]
First, since $p$ depends only on $q_1$ and $q_2$, we may assume that $q_2 = q_3 = \dotsb = q_r$. 
We say that a coloring is proper if it contains no monochromatic $q_i$-AP colored $i$; we wish to prove that a.a.s.\ $[n]_p$ admits no such coloring. \par
By our assumption, $p \ge Cn^{-\frac{q_2}{q_1(q_2-1)}}$ for some sufficiently large $C >0$.
Since not admitting a proper coloring is an increasing event, without loss of generality we may assume that $p=Cn^{-\frac{q_2}{q_1(q_2-1)}}$.\par
Note that the elements of the set $\{a \in [n]_p : a \text{ does not belong to a $q_1$-AP}\}$ may all be recolored $1$ for any proper coloring of $[n]_p$, without creating a monochromatic $q_1$-AP. We say that such a proper coloring is a ``good'' coloring. Since, by recoloring, the existence of a proper coloring implies the existence of a good coloring it suffices to show that no good coloring exists.\par 
Suppose that there is such a coloring; then for each $i \in [r]$, the set $I_i$ of elements colored $i$ contains no $q_i$-APs. 
By \nameref{thm:hypergraph-container} for every $\epsilon >0$ and every $i \geq 2$ there exist $S_i \subset I_i \subset G_i$ with $\abs{S_i} \leq s_{\max} = O(n^{1-\frac{1}{q_2-1}})$ and such that $G_i$, which depends only on $S_i$, contains at most $\epsilon n^2$ many $q_2$-APs (the implicit constant in the definition of $s_{\max}$ may depend on $\epsilon$ and $q_2$).
By our assumption, $$I_1 = [n]_p \setminus (I_2 \cup \dotsb \cup I_r) \supset [n]_p \setminus (G_2 \cup \dotsb \cup G_r) = [n]_p \cap ([n]\setminus (G_2 \cup \dotsb \cup G_r)).$$
In particular,
\begin{gather*}
    I_1 \text{ contains no $q_1$-APs} \implies  
    [n]_p \cap ([n]\setminus (G_2 \cup \dotsb \cup G_r)) \text{ contains no $q_1$-APs}.
\end{gather*}
For brevity, we write $A_{(S_2,\dotsc,S_r)} = [n]_p \cap \left([n]\setminus (G_2 \cup \dotsb \cup G_r)\right)$, as $G_i$ depends only on $S_i$.
Since $q_1 > q_2$, the number of $q_1$-APs in any set of integers is at most as large as the number of $q_2$-APs.
Therefore, for every choice of $(S_2, \dotsc, S_r)$, Lemma \ref{lemma:AP-supersaturation} implies that $[n]\setminus (G_2 \cup \dotsb \cup G_r)$ contains at least $\epsilon n^2$ many $q_1$-APs.
Hence, by Lemma \ref{lemma:no-q-ap-janson-bound} 
$$\mathbb{P}(A_{(S_2,\dotsc,S_r)} \text{ contains no $q_1$-AP}) \leq e^{-D_1n^2p^{q_1}}$$
for some constant $D_1 > 0$ which depends only on $\epsilon$ and $q_1$.\par
Let $Q$ and $Q_I$ be the variables defined in Definition \ref{def:mostly-independent}.
Suppose $S = \cup_{i=2}^rS_i$ is covered by elements of $Q$; we fix a largest subset of $S$ that is covered by pairwise-disjoint arithmetic progressions of length $q_1$ in $[n]_p$ and denote it $S'$.
Since $S \setminus S' \subset S \cap (Q \setminus Q_I)$, if $Q$ is mostly independent then 
$ \abs{S\setminus S'} < n^{1-\frac{1}{q_2-1}-\delta} $.
Thus, if a good coloring exists and $Q$ is mostly independent then there exists some choice of  $(S_2, \dotsc, S_r)$ such that 
\begin{enumerate}
    \item $S$ is covered by $q_1$-APs in $[n]_p$.
    \item $\abs{S \setminus S'} < n^{1-\frac{1}{q_2-1}-\delta}$.
    \item $A_{(S_2,\dotsc,S_r)}$ contains no $q_1$-AP.
\end{enumerate}
For shorthand, we say that $S$ is ``well-covered" if it satisfies conditions 1 and 2.\par
Let $P = \mathbb{P}([n]_p\text{ admits a ``good" coloring})$.
We wish to show that $P = o(1)$.
We first note $P$ may be bounded by the sum of probabilities of two other events: either $Q$ is not mostly independent or there exists a tuple $(S_2, \dotsc, S_r)$ such that the above three events hold.
Since by Lemma~\ref{claim:APs-are-mostly-indpendent} the probability that $Q$ is not mostly independent is $o(1)$ we obtain 
\begin{gather*}
P
 \leq o(1) + 
 \sum_{(S_2,\dotsc,S_r)}\mathbb{P}(S \text{ is well-covered }  \land \ A_{(S_2,\dotsc,S_r)}  \text{ contains no $q_1$-AP}) .
\end{gather*}
Note, that the event ``$S$ is well-covered" is increasing, while the event ``$A_{(S_2,\dotsc,S_r)}$ contains no $q_1$-AP" is decreasing. Therefore by Harris's inequality we obtain
\begin{align*}
P
 &\leq o(1) + 
 \sum_{(S_2,\dotsc,S_r)}\mathbb{P}(S \text{ is well-covered})  \cdot \mathbb{P}( A_{(S_2,\dotsc,S_r)}  \text{ contains no $q_1$-AP}) \\
 &\leq o(1) + \sum_{(S_2,\dotsc,S_r)}\mathbb{P}(S \text{ is well-covered})  \cdot  e^{-D_1n^2p^{q_1}}.
\end{align*}
Our goal is now to obtain a bound on \[\sum_{(S_2,\dotsc,S_r)}\mathbb{P}(S \text{ is well-covered}).\] 
Let $\mathcal{C}(A)$ denote the event ``$A$ is covered by pairwise-disjoint $q_1$-APs".
Since there are at most $r^{\abs{S \setminus S'}}$ many ways to distribute the elements of $S \setminus S'$ to $S_2,\dotsc,S_r$,
we obtain
\begin{gather*}
\sum_{(S_2,\dotsc,S_r)}\mathbb{P}(S \text{ is well-covered}) 
\leq \sum_{t=0}^{n^{1-\frac{1}{q_2-1}-\delta}}\binom{n}{t}r^t \cdot \sum_{(S_2',\dotsc,S_r')}\mathbb{P}(\mathcal{C}(S')),
\end{gather*}
where the second sum ranges over all $(r-1)$-tuples of sets $(S'_2, \dotsc, S'_r)$ satisfying $\abs{S'_i} \le s_{\max}$ for each $i$ and $S' = \cup_{i=2}^rS'_i$;
which we may bound from above by
\[
    O(e^{3\log n\cdot n^{1-\frac{1}{q_2-1} - \delta}}) \cdot \sum_{(S_2',\dotsc,S_r')}\mathbb{P}(\mathcal{C}(S')).
\]
We now move on to bound \[\sum_{(S_2',\dotsc,S_r')}\mathbb{P}(\mathcal{C}(S')).\]
Suppose $S' = \cup_{i=2}^rS'_i$ is fixed, then there are $(r-1)^{\abs{S'}}$ many ways to distribute its elements into $r-1$ different subsets; hence,
\begin{gather*}
    \sum_{(S'_2,\dotsc,S'_r)}\mathbb{P}(\mathcal{C}(S'))  
    \leq O(2^{rs_{\max}}) \sum_{\abs{S'} < rs_{\max}} \mathbb{P}(\mathcal{C}(S')).
\end{gather*} \par
We note that $\sum_{\abs{S'} < rs_{\max}} \mathbb{P}(\mathcal{C}(S'))$ is simply the expected number of sets of size at most $rs_{\max}$ that are covered by pairwise-disjoint $q_1$-APs.
This, in turn, may be bounded from above by 
\[
\sum_{N \le rs_{\max}} \sum_{s < rs_{\max}}\binom{q_1N}{s} \mathbb{E}(\#\text{collections of }N\text{ pairwise-disjoint $q_1$-APs in $[n]_p$}). \]
Since the expected number of choices for collections of $N$ pairwise-disjoint $q_1$-APs in $[n]_p$ is at most $\frac{(n^2p^{q_1})^N}{N!}$, we conclude
\begin{align*}
    \sum_{\abs{S'} < rs_{\max}} \mathbb{P}(\mathcal{C}(S')) 
    &\le \sum_{N \le rs_{\max}} \sum_{s < rs_{\max}} \frac{(n^2p^{q_1})^N}{N!} \binom{q_1N}{s} \\
    &\le \sum_{N \le rs_{\max}} \frac{(n^2p^{q_1})^N}{N!} 2 ^{q_1N} \\
    &\le \sum_{N \le rs_{\max}} \left(\frac{e2^{q_1}n^2p^{q_1}}{N}\right)^N.
\end{align*}
Note, that $N \le rs_{\max} \le D_2' n^{1 - \frac{1}{q_2-1}}$ for some $D_2' > 0$ which does not depend on~$C$.
Since $x \to \left(\frac{ea}{x}\right)^x$ grows for $x \leq a$ and $n^2p^q_1 = C^{q_1} n^{1 - \frac{1}{q_2 - 1}}$, for $C$ large enough we have
\[
    \left(\frac{e2^{q_1}n^2p^{q_1}}{N}\right)^N \le \left(\frac{(e2C)^{q_1}}{D_2'}\right)^{D_2' n^{1 - \frac{1}{q_2 - 1}}} \le e^{D_2' q_1 \log (2eC) n^{1 - \frac{1}{q_2 - 1}}}.
\]
Hence,
\begin{align*}
    \sum_{N < rs_{\max}}  \left(\frac{e2^{q_1}n^2p^{q_1}}{N}\right)^N  
    \le \sum_{N < rs_{\max}}  e^{D_2' q_1 \log (2eC) n^{1 - \frac{1}{q_2 - 1}}} \le e^{O(\log C n^{1 - \frac{1}{q_2 - 1}})}.
\end{align*}
Finally we obtain, 
\begin{align*}
    \sum_{(S_2,\dotsc,S_r)}\mathbb{P}(S \text{ is well-covered}) 
    &\le O(e^{O(n^{1 - \frac{1}{q_2 - 1}} \log C + \log n\cdot n^{1-\frac{1}{q_2-1} - \delta})}) \\ &\le e^{D_2 \log C n^{1-\frac{1}{q_2-1}}}
\end{align*}
for some constant $D_2 > 0$ which does not depend on C. \\
Therefore, for $C$ large enough, 
\begin{align*}
    P
    &\leq o(1) + \mathbb{P}(S \text{ is well-covered})\cdot e^{-D_1n^2p^{q_1}} \\
  &= o(1) + O(e^{D_2 \log Cn^{1 - \frac{1}{q_2 - 1}})}) \cdot O(e^{-D_1C^{q_1}n^{1 - \frac{1}{q_2 - 1}}}) \to 0
\end{align*}
\end{proof}
\section{The asymmetric case 0-statement} \label{sec:asymm-0}
In this section, we show that for any integers $q_1 > q_2$ there exists a sufficiently small positive $c > 0$ such that if $p=c \cdot n ^{-\frac{q_2}{q_1(q_2-1)}}$ the elements of $[n]_p$ can a.a.s.\ be colored red/blue without a monochromatic $q_1$-AP colored red or a monochromatic $q_2$-AP colored blue.
We note that this is sufficient for the 0-statement for any other number of colors.
We begin by making several definitions that will assist us in stating the results of this section in the language of hypergraphs.
\begin{definition}
    Throughout this section we will deal with hypergraphs with edges of two possible cardinalities $q_1$ and $q_2$, we call such hypergraphs $(q_1,q_2)$-uniform. We will refer to edges as \emph{long} or \emph{short} edges, depending on their cardinalities.
    We say that a $(q_1,q_2)$-uniform hypergraph is \emph{asymmetrically-2-colorable} if its vertices can be colored red/blue with no long edge colored red, and no short edge colored blue. 
\end{definition}
\begin{definition}
Let $H(n,q_1,q_2)$ be the hypergraph with vertex set $V(H) = [n]$, whose edge set is the set of arithmetic progressions of lengths $q_1$ and $q_2$.
We denote by $H(n,q_1,q_2, p)$ the random subhypergraph of $H(n,q_1,q_2)$ induced by $[n]_p$. 
\end{definition}
With these definitions in hand, we are ready to state this section's main result:
\begin{theorem}[Asymmetric 0-statement] 
\label{thm:asymmetric-0-statement} 
For any integers $q_1 > q_2 \geq 3$, there exists $c>0$ such that for $p\le c \cdot n^{-\frac{q_2}{q_1(q_2-1)}}$, $H(n,q_1,q_2,p)$ is asymmetrically-2-colorable a.a.s.
\end{theorem}

The proof we present here is similar in nature to the proof by Rödl and Ruci\'{n}ski~\cite{RR97} of the (symmetric) random Rado partition theorem and consists of two main lemmas.
First, in Lemma \ref{lemma:asymmetric-deterministic-2-coloring}, we show that a $(q_1,q_2)$-uniform hypergraph is asymmetrically-2-colorable unless it contains a member of a small family of hypergraphs which we call $2$-blocking. 
Then, Lemma \ref{lemma:asymmetric-probabilistic-lemma} states that those hypergraphs a.a.s.\ do not appear in $H(n,q_1,q_2,p)$. 
We begin by making several definitions.

\begin{definition}
    We say that an edge $E = \{a_1,\dotsc,a_{q}\}$ has a \emph{cover} if there are edges $E_1,\dotsc,E_{q}$ such that $E \cap E_i = \{a_i\}$ and $\abs{E_i} \ne \abs{E}$ for all $i \in [q]$.
    We say that a cover for an edge is \emph{simple} if $E_i \cap (\bigcup_{j \ne i} E_j) = \emptyset$ for all $i \in [q]$.
    If every edge in a hypergraph $H$ has a cover, we say $H$ is \emph{covered}. 
\end{definition}
\begin{definition}
    A \emph{simple path} of length $\ell$ is a hypergraph consisting of short edges $E_1, \dotsc,E_\ell$, and covering long edges $E_{1,1}, \dotsc, E_{1,q_2}, E_{2,1}, \dotsc, E_{\ell, q_2}$ such that $E_{i,1}, \dotsc, E_{i,q_2}$ cover $E_i$ and such that:
    \begin{enumerate}
        \item $E_{i+1,1} = E_{i,q_2}$ for every $i < \ell$,
        \item no two edges of the same cardinality intersect.
    \end{enumerate}
    A simple path of length one is called a \emph{block}; thus, a simple path consists of blocks, such that every pair of consecutive blocks share a long edge. 
    For convenience, we refer to a single long edge as a simple path of length zero.
\end{definition}
\begin{definition}
    We say that a simple path $P$ of length $\ell$ has a \emph{saw} if for every $v_i \in E_{1,1} \setminus E_1$ there exists a short edge $S_i$ such that $S_i \cap E_{1,1} = \{v_i\}$ and $\abs{S_i \cap V(P)} = 2$. 
    We call the edges $S_i$ the \emph{saw edges} for $P$.
\end{definition}
\begin{definition}
    We say that a simple path $P$ of length $\ell$ is \emph{spoiled} if there exists an edge $E \notin E(P)$ such that $\abs{E \cap V(P)} \ge 3$ and $E \cap E_{\ell, q_2} = \{v\}$ for some $v \notin E_\ell$.
\end{definition}
\begin{definition}
    We say that a simple path $P$ of length $\ell$ has a \emph{spoiled extension}, if there exists a short edge $E_{\ell+1}$ along with a simple cover $E_{\ell+1,1}, \dotsc, E_{\ell+1,q_2}$, such that $E_{\ell+1} \cap V(P) = \{v\}$ for some $v \in E_{\ell, q_2} \setminus E_\ell$, $E_{\ell+1,1} = E_{\ell,q_2}$ and there exists $i \in \{2, \dotsc, q_2\}$ such that $E_{\ell+1,i} \cap V(P) \ne \emptyset$.
\end{definition}

\begin{figure}[h]
\centering
\begin{subfigure}[b]{.30\textwidth}
\centering
\includegraphics[width=\textwidth]{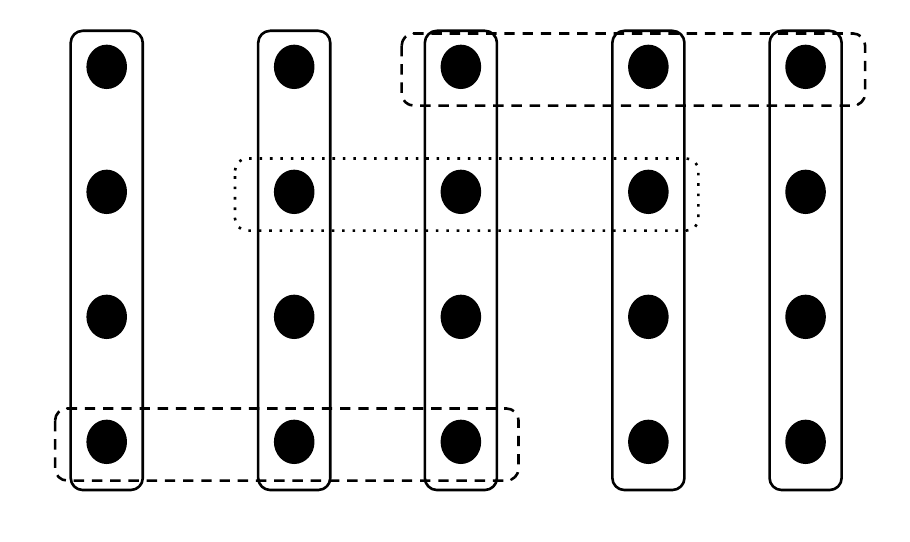}
\caption{A spoiled simple path of length 2, with the spoiling edge dotted.}
\end{subfigure}
\hfill
\begin{subfigure}[b]{.30\textwidth}
\centering
\includegraphics[width=\textwidth]{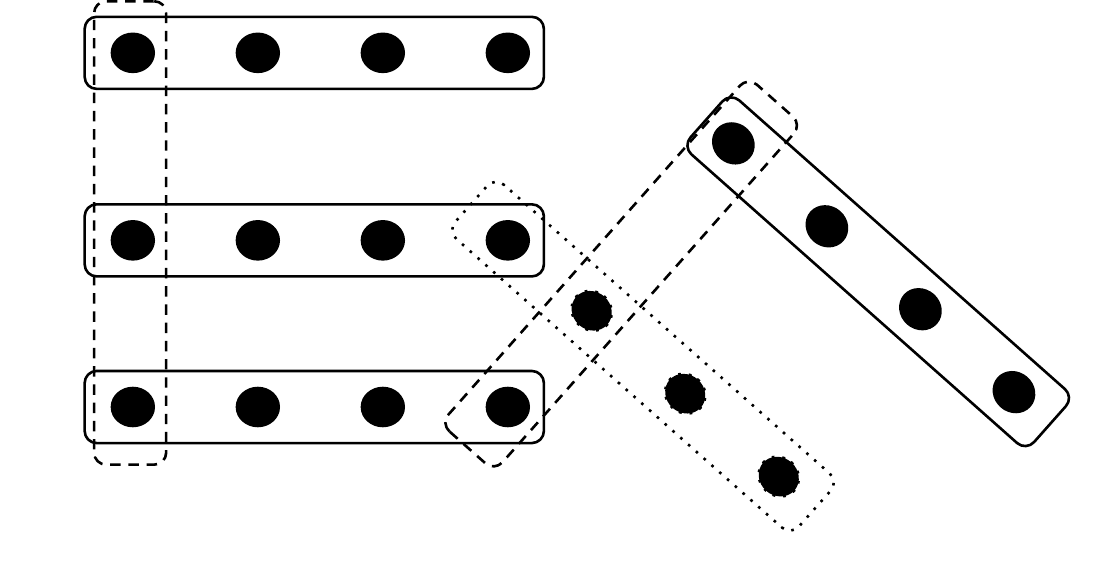}
\caption{A block with a spoiled extension, with the spoiling long edge dotted.}
\end{subfigure}
\hfill
\begin{subfigure}[b]{.30\textwidth}
\centering
\includegraphics[width=\textwidth]{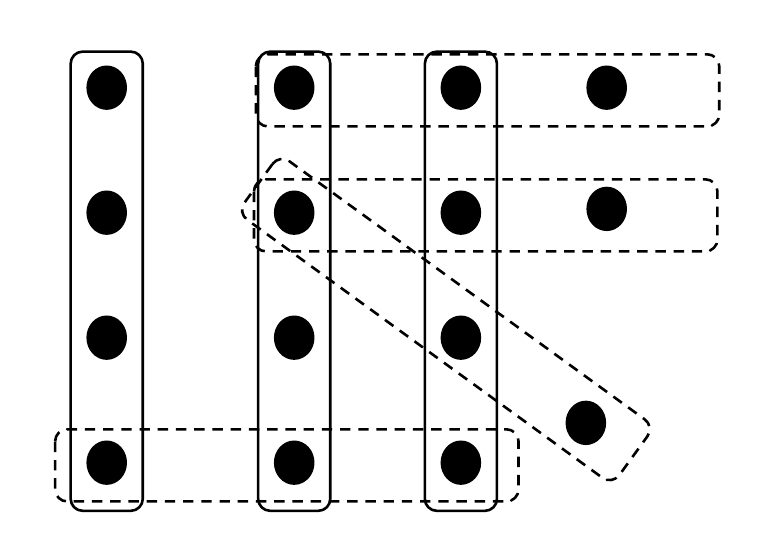}
\caption{A path of length one with a saw.}
\end{subfigure}
\label{fig:asym-hypergraphs}
\caption{Examples of 2-blocking hypergraphs for $q_1=4, q_2=3$.}
\end{figure}

\begin{definition}
    We say that a $(q_1,q_2)$-uniform hypergraph is \emph{2-blocking} if it is one of the following:
    \begin{enumerate}
        \item A short edge with a non-simple cover.
        \item A spoiled simple path.
        \item A simple path with a saw.
        \item A simple path with a spoiled extension.
    \end{enumerate}
\end{definition}
\begin{lemma}[The deterministic lemma]\label{lemma:asymmetric-deterministic-2-coloring}
    Let $q_1 > q_2 \geq 3$ be some integers, and let $H$ be a $(q_1,q_2)$-uniform hypergraph which is not asymmetrically-2-colorable.
Then $H$ contains a 2-blocking hypergraph. 
\end{lemma}

\begin{lemma}[The probabilistic lemma]
\label{lemma:asymmetric-probabilistic-lemma}
Let $H = H(n, q_1, q_2, p)$, let $c$ be a sufficiently small positive constant, and let $p=c \cdot n^{-\frac{q_2}{q_1(q_2-1)}}$.
Then a.a.s.\ $H$ contains no 2-blocking hypergraph.
\end{lemma}

Theorem \ref{thm:asymmetric-0-statement} follows immediately from these two lemmas.
Most of this section deals with proving the probabilistic lemma, using a first-moment argument over several random variables.
But first, we begin by proving the deterministic part of the theorem.
\begin{proof}[Proof of \nameref{lemma:asymmetric-deterministic-2-coloring}]
We say that a hypergraph is edge-critical if it is not asymmetrically-2-colorable, but any proper subhypergraph is. 
We may assume that $H$ is edge-critical; otherwise, we replace it with an edge-critical subhypergraph. We begin by showing that every edge critical hypergraph is covered. 
\begin{claim}\label{claim:asym-deterministic-lemma}
If $H$ is an edge-critical hypergraph, then for every edge $E \in H$ and for every vertex $v \in E$ there is an edge $E' \in H$ such that $E \cap E' = \{v\}$ and $\abs{E} \neq \abs{E'}$; in other words, $H$ is covered.
\end{claim}
\begin{proof}
Let $H$ be edge-critical, and suppose that there are an edge $E \in H$ and a vertex $v \in E$, such that every edge $E'$ of the other cardinality that contains $v$ also contains another vertex of $E$.
By the edge-criticality, $H$ can be colored red/blue in such a way that only $E$ violates the coloring condition. Without loss of generality, assume $E$ is long (and colored red). Now, by changing the color of $v$ to blue, $E$ no longer violates the coloring condition, and neither does any short edge that contains $v$, contradicting the fact that $H$ is not asymmetrically-2-colorable.
\end{proof}
If there exists a short edge with a non-simple cover, then we are done, so we may assume all short edges have simple covers.
Let $P$ be a longest simple path, and let $\ell \ge 1$ be its length.
We observe $E = E_{\ell,q_2}$; since $E$ is covered by short edges, we have short edges $S_i$ such that $S_i \cap E = \{v_i\}$ for every $v_i \in \{v_2, \dotsc, v_{q_1}\} =  E \setminus E_\ell$.  \par
Suppose first that there exists $S_i$ such that $S_i \cap V(P) = \{v_i\}$ and observe $\{s_2, \dotsc, s_{q_2}\} = S_i \setminus E$.
By the previous claim, we have a long covering edge for every $s_i$, and by the maximality of $P$ we obtain that for every simple cover for $S_i$ there exists a covering edge $L$ that intersects $P$. 
Thus, we obtain that $S_i$ forms a spoiled extension to $P$. \par
If, on the other hand, no $S_i$ intersects $P$ in exactly a single vertex, we obtain one of two cases:
If there exists some $i$ such that $\abs{S_i \cap V(P)} \ge 3$, then since $\abs{S_i \cap E} = 1$ we obtain that $S_i$ is a spoiling edge for $P$.
Otherwise, we have $\abs{S_i \cap V(P)} = 2$ for all $i$ and thus obtain a path with a saw, completing the proof.
\end{proof}

The rest of this section deals with proving the probabilistic portion of Theorem~\ref{thm:asymmetric-0-statement}. We will begin by proving upper bounds on the number of copies of several hypergraphs in $H(n,q_1,q_2)$ and showing that a.a.s.\ all short edges have simple covers.
We will then prove an upper bound on the number of simple paths of arbitrary lengths. 
Finally, with the above results in hand, we will turn to prove Lemma~\ref{lemma:asymmetric-probabilistic-lemma}.

\begin{lemma}\label{lemma:single-edge-multicover}
    Let $x,y \in [n]$ be distinct integers. Then, the number of choices for $a,b \in [n]$ such that there are $q$-APs that contain $\{x,a,b\}$ and $\{y,a,b\}$ is $O(1)$.
\end{lemma}
\begin{proof}
    Suppose that $a,b$ are contained in a $q$-AP along with $x$, then there exists $t_1 \in \mathbb{Q}$ such that $x-a = t_1(b-a)$, thus, $x = (1 - t_1)a + t_1b$.
    Moreover, $t_1 = \frac{r_1}{r_2}$ for $r_1, r_2 \in \{-q, -q+1,\dotsc, q\}$.
    The same also holds for $y$ with another constant $t_2$.
    We obtain the following system of linear equations:
\[
    \begin{pmatrix} 1-t_1&t_1 \\ 1-t_2& t_2 \end{pmatrix} \begin{pmatrix} a \\ b \end{pmatrix} = \begin{pmatrix} x \\ y \end{pmatrix}. \tag{$\star$} \label{eq:mc}
\]
Since the determinant of the above matrix is $t_2 - t_1$, we obtain that as long as $t_1 \ne t_2$ there is only a single solution to $\eqref{eq:mc}$.
Since $x \ne y$ implies $t_1 \ne t_2$ and there are at most $O(q^2) = O(1)$ choices for $t_1$ and $t_2$, there are only $O(1)$ many choices for $a,b$.
\end{proof}

In Lemma \ref{lemma:minimal-covers} we show that a short edge with its $q_2$ covering long edges must contain almost $2q_1$ vertices. We will then use this lemma to show that non-simple covers are unlikely. We will require the following two elementary lemmas.

\begin{lemma}\label{lemma:AP-intersection}
    Let $E_1, E_2$ be two arithmetic progressions of length $q_1$, with common differences $d_1$ and $d_2$. 
    If $d_1 < d_2$, then $\abs{E_1 \cap E_2} \le \ceil{q_1 \cdot \frac{\gcd(d_1,d_2)}{d_2}}$.
\end{lemma}
\begin{proof}
    Let $A_1$ and $A_2$ be the infinite arithmetic progressions containing $E_1$ and $E_2$ respectively.
    We obtain that $A = A_1 \cap A_2$ is either empty or an infinite arithmetic progression with common difference $\lcm(d_1,d_2)$, thus $A \cap A_1$ contains every $\frac{\lcm(d_1,d_2)}{d_1}$-th element of $A_1$.
    Therefore, a subsequence of length $q_1$ in $A_1$ contains at most $\lceil q_1 \cdot \frac{d_1}{\lcm(d_1,d_2)} \rceil$ elements of $A$, and the result follows immediately.
\end{proof}

\begin{lemma}\label{lemma:AP-modulo}
    Let $n > m >0$ be integers, and $A = (a_1, a_2, \dotsc, a_q)$ be an arithmetic progression of length $q > 3$ with common difference $m$.
    If we denote $t = \frac{n}{\gcd(n,m)}$, then $\abs{\{a \in [n] : \exists i \in [q] \text{ such that } a \equiv a_i \mod n\}} = \min\{t,q\}$.
    Moreover, $a_i \equiv a_{i+kt} \mod n$ for all integers $i$ and $k$ such that $i, i + kt \in [q]$.
\end{lemma}
\begin{proof}
Let $G = \mathbb{Z}/n\mathbb{Z}$ be the additive cyclic group of order $n$. 
From elementary group theory we know that the order of $m$ in $G$ is $t = \frac{n}{\gcd(n,m)}$. Let $G'$ be the cyclic subgroup generated by $m$.
Then, the residues of $A$ modulo $n$ are contained in the coset $a_1 + G'$, which has $\abs{a_1 + G'} = \abs{G'} = t$, completing the proof.
\end{proof}

\begin{lemma}\label{lemma:minimal-covers}
    If $E$ is a short edge and $E_1,\dotsc, E_{r}$ are a subset of its covering edges, then 
    \[\abs{E_1 \cup \dotsb \cup E_{r}} > 2q_1\left(1-\frac{1}{r}\right).\]
\end{lemma}
\begin{proof}
    Denote $M = \abs{E_1 \cup \dotsb \cup E_{r}}$ and assume for contradiction that $M \le 2q_1(1 - \frac{1}{r})$.
    The case $r \in \{1,2\}$ is trivial, therefore we may assume $r \ge 3$.
    We first show that no three covering edges share the same common difference.
    \begin{claim}
        Let $E_1, E_2, E_3$ be covering edges, and assume that all three $q_1$-APs have the same common difference, i.e.\ $E_z = \{a_z + i \cdot d : i \in [q_1]\}$.
        Then, \[\abs{E_1 \cup E_2 \cup E_3} \ge 2q_1.\]
    \end{claim}
    \begin{proof}
        Suppose $\{v_z\} = E_z \cap E$ for $z \in \{1,2,3\}$.
        Without loss of generality assume that $v_1 < v_2 < v_3$, and that $E_z \cap E_2 \ne \emptyset$ for $z=1,3$, as otherwise $\abs{E_z \cup E_2} = 2q_1$.
        Since each covering edge may only contain one vertex of $E$, and $E_1$ lies on the same infinite arithmetic progression of difference $d$ as $E_2$,  we deduce that all the elements of $E_1$ must be strictly smaller than $v_2$. Similarly, all elements of $E_3$ must be strictly larger than $v_2$.
        Hence, $E_1 \cap E_3 = \emptyset$; and thus, $\abs{E_1 \cup E_3} = 2q_1$.
    \end{proof}
    We will now show that $r \ge 5$.
    Since $r \ge 3$, the previous claim implies that there exists a pair of edges with different common differences; without loss of generality we denote them $E_1, E_2$.
    We note that $\abs{E_1 \cap E_2} \le \lceil\frac{q_1}{2}\rceil$, by Lemma \ref{lemma:AP-intersection}. 
    Let $E' = E \cap (E_1 \cup \dotsb \cup E_r)$.
    Since $E_1,E_2$ are covering edges of $E$, we have 
    \begin{align*}
        \abs{E' \cup E_1 \cup E_2} = \abs{E_1 \cup E_2} + \abs{E'} - 2 &= \abs{E_1} + \abs{E_2} - \abs{E_1 \cap E_2} + \abs{E'} - 2 \\&\ge q_1 + \floor*{\frac{q_1}{2}} + r - 2.
    \end{align*}
    Thus, if $r \in  \{3,4\}$ we observe that (since $\lfloor \frac{q_1}{2} \rfloor \ge \frac{q_1}{2} - \frac{1}{2}$)
    \[
        q_1 + \floor*{\frac{q_1}{2}}+ r - 2 \ge
        \frac{3}{2}q_1 + r - \frac{5}{2} \ge 
        \left(1-\frac{1}{r}\right)2q_1 + r - \frac{5}{2} >
        2q_1\left(1 - \frac{1}{r}\right)  
    \] for all $q_1 > r$.
    Finally, we will show that no three arithmetic progressions may have pairwise different common differences.
    \begin{claim}
        Let $E_1, E_2, E_3$ be covering edges, with distinct common differences $d_1, d_2, d_3$ respectively. 
        Then 
        \[\abs{E_1 \cup E_2 \cup E_3} \ge 2q_1 - 4,\]
        and thus $M \ge \abs{E_1 \cup E_2 \cup E_3} +(r-3)  \ge 2q_1 - 2> 2q_1(1-\frac{1}{r})$.
    \end{claim}
    \begin{proof}
        Assume for contradiction that $\abs{E_1 \cup E_2 \cup E_3} < 2q_1 - 4$.
        Since \[\abs{E_1 \cup E_2 \cup E_3} \ge 3q_1 - \abs{E_1 \cap E_2} - \abs{E_1 \cap E_3} - \abs{E_2 \cap E_3},\] at least one pair of edges intersects in more than $\frac{q_1}{3} + 1$ elements. \\
        Recall that $q_1 > q_2 \ge r \ge 5$. 
        Note that $\frac{q_1}{3} + 1 >  \lceil \frac{q_1}{3} \rceil \ge \lceil\frac{q_1}{d}\rceil$ for all $d \ge 3$ and $q_1 \ge 6$.
        Thus, by Lemma \ref{lemma:AP-intersection}, we have that for some $i,j \in [3]$  such that $d_i > d_j$ we have $\gcd(d_i,d_j) = \frac{d_i}{2}$ which can only occur if $d_i = 2d_j$.
    Let $k \notin \{i,j\}$. Since 
    \[\abs{E_i \cup E_j \cup E_k} \ge 3q_1 - \ceil*{\frac{q_1}{2}} - \abs{E_i \cap E_k} - \abs{E_j \cap E_k} \ge 3q_1 - \frac{q_1}{2} - 1 - \abs{E_i \cap E_k} - \abs{E_j \cap E_k}\]
    we obtain that $E_k$ must intersect one of the other edges in more than $\frac{q_1}{4} + 1$ elements. 
    Note that $\frac{q_1}{4} + 1 >  \lceil \frac{q_1}{4} \rceil \ge \lceil\frac{q_1}{d}\rceil$ for all $d \ge 4$ and $q_1 \ge 6$.
    Again, by Lemma \ref{lemma:AP-intersection}, we have that $t \cdot \gcd(d_k, d_z) = \max(d_k,d_z)$ for some $z \in \{i,j\}$ and $t \in \{2,3\}$. \par
    If we set $d_j = d$, we get that $\frac{d_k}{d}$ may obtain one of six values: either $\frac{1}{2}$ or $4$ for $t = 2$ or one of $\frac{1}{3}, \frac{2}{3}, 3, 6$ for $t=3$.
    We note that by Lemma \ref{lemma:AP-modulo}, if we observe the residues modulo $d'$ of an arithmetic progression of length $q_1$ with common difference $d < d'$, we obtain that the progression cycles through $\frac{d'}{\gcd(d,d')}$ residues; and thus, it contains at least $\left\lfloor\frac{q_1 \cdot \gcd(d,d')}{d'}\right\rfloor$ elements of every residue class it encounters.
    By adjusting the constant $d$ and permuting the indices we obtain that the triplet $(d_1, d_2, d_3)$ must fall into one of five categories:
        \begin{enumerate}
            \item Assuming $(d_1, d_2, d_3) = (d, 2d, 4d)$, without loss of generality we assume $0 \in E_3$ and $d=1$.
                Then, $E_3$ consists of $q_1$ elements which satisfy $a \equiv 0 \mod 4$.
                If $E_2$ contains elements which satisfy $a \equiv 1,3 \mod 4$, we obtain that $\abs{E_3 \cup E_2} = 2q_1$. \\
                Otherwise, we obtain that at least $\lfloor \frac{q_1}{2} \rfloor$ of the elements in $E_2$ satisify $a \equiv 2 \mod 4$, and at least $2\cdot\lfloor \frac{q_1}{4} \rfloor$ of the elements in $E_1$ satisfy $a \equiv 1,3 \mod 4$.
                Thus, \[\abs{E_1 \cup E_2 \cup E_3} \ge q_1 + \lfloor \frac{q_1}{2} \rfloor + 2\cdot\lfloor \frac{q_1}{4} \rfloor \ge 2q_1 - 2.\]
            \item Assuming $(d_1, d_2, d_3) = (d, 2d, 3d)$, without loss of generality we assume $0 \in E_3$ and $d=1$.
                Then, $E_3$ consists of $q_1$ elements which satisfy $a \equiv 0, 3 \mod 6$.
                Since at least $2\lfloor \frac{q_1}{3} \rfloor$ of the elements in $E_2$ satisify $a \equiv 2,4 \mod 6$ (or equivalently $1,5 \mod 6$), and at least $2\lfloor \frac{q_1}{6} \rfloor$ of the elements in $E_1$ satisfy $a \equiv 1,5 \mod 6$ (equivalently $2,4 \mod 6$), we obtain
                \[\abs{E_1 \cup E_2 \cup E_3} \ge q_1 + 2\lfloor \frac{q_1}{3} \rfloor + 2\lfloor \frac{q_1}{6} \rfloor \ge 2q_1 - 3.\]

            \item Assuming $(d_1, d_2, d_3) = (d, 2d, 6d)$, without loss of generality we assume $0 \in E_3$ and $d=1$.
                Then, $E_3$ consists of $q_1$ elements which satisfy $a \equiv 0 \mod 6$.
                If $E_2$ contains elements which satisfy $a \equiv 1,3,5 \mod 6$ we obtain that $\abs{E_3 \cup E_2} = 2q_1$. \\
                Otherwise, at least $2\lfloor \frac{q_1}{3} \rfloor$ of the elements in $E_2$ satisify $a \equiv 2,4 \mod 6$, and at least $3\lfloor \frac{q_1}{6} \rfloor$ of the elements in $E_1$ satisfy $a \equiv 1,3,5 \mod 6$. 
                Hence, \[\abs{E_1 \cup E_2 \cup E_3} \ge q_1 + 2\lfloor \frac{q_1}{3} \rfloor + 3\lfloor \frac{q_1}{6} \rfloor \ge 2q_1 - 4.\]
            \item Assuming $(d_1, d_2, d_3) = (d, 3d, 6d)$, without loss of generality we assume $0 \in E_3$ and $d=1$.
                Then, $E_3$ consists of $q_1$ elements which satisfy $a \equiv 0 \mod 6$.
                If $E_2$ contains elements which satisfy $a \equiv 1,4 \mod 6$ or $a \equiv 2,5 \mod 6$ we obtain that $\abs{E_3 \cup E_2} = 2q_1$. \\
                Otherwise, at least $\lfloor \frac{q_1}{2} \rfloor$ of the elements in $E_2$ satisify $a \equiv 3 \mod 6$, and at least $4\lfloor \frac{q_1}{6} \rfloor$ of the elements in $E_1$ satisfy $a \equiv 1,2,4,5 \mod 6$. 
                Therefore, \[\abs{E_1 \cup E_2 \cup E_3} \ge q_1 + 4\lfloor \frac{q_1}{6} \rfloor + \lfloor \frac{q_1}{2} \rfloor  \ge 2q_1 - 4.\]
            \item Assuming $(d_1, d_2, d_3) = (2d, 3d, 6d)$, without loss of generality we assume $0 \in E_3$ and $d=1$.
                Then, $E_3$ consists of $q_1$ elements which satisfy $a \equiv 0 \mod 6$.
                If $E_2$ contains elements which satisfy $a \equiv 1,4 \mod 6$ or $a \equiv 2,5 \mod 6$ we obtain that $\abs{E_3 \cup E_2} = 2q_1$. \\
                If $E_1$ contains elements which satisfy $a \equiv 1,3,5 \mod 6$ we obtain that $\abs{E_3 \cup E_1} = 2q_1$. \\
                Otherwise, at least $\lfloor \frac{q_1}{2} \rfloor$ of the elements in $E_2$ satisify $a \equiv 3 \mod 6$, and at least $2\lfloor \frac{q_1}{3} \rfloor$ of the elements in $E_1$ satisfy $a \equiv 2,4 \mod 6$.
                Thus, \[\abs{E_1 \cup E_2 \cup E_3} \ge q_1 + 2\lfloor \frac{q_1}{3} \rfloor + \lfloor \frac{q_1}{2} \rfloor  \ge 2q_1 - 2.\]
        \end{enumerate}
        \end{proof}
        Finally we obtain that $r \ge 5$, no three edges may share a common difference, and there can be no more than two possible common differences which yields a contradiction.
    \end{proof}

\begin{lemma}\label{lemma:only-simple-covers}
    Let $H'$ be a hypergraph consisting of a short edge and its non-simple cover. If $p = O(n^{-\frac{q_2}{q_1(q_2-1)}})$, then $\mathbb{P}(H' \subset H(n,q_1,q_2,p)) = o(1)$.
\end{lemma}
\begin{proof}
    The following computations are used several times throughout the proof: \[p^{q_1-1} = O(n^{-\frac{q_2}{q_2-1} + \frac{q_2}{q_1(q_2-1)}}) = O(n^{-1 - \frac{1}{q_2-1} + \frac{q_2}{q_1(q_2-1)}}) = O(n^{-1 - \frac{q_1 - q_2}{q_1(q_2-1)}}),\]
    and
    \[ p^{q_1} = O(n^{-1 - \frac{1}{q_2-1}}) .\]

    Let $E$ be a short edge and let $E_1, \dotsc, E_{q_2}$ be its covering edges. We divide the covering edges into three categories: we say that a covering edge $E_i$ is simple if $\abs{E_i \cap (\bigcup_{j \ne i} E_j)} = 0$, we say a covering edge $E_i$ is generic if $\abs{E_i \cap (\bigcup_{j \ne  i} E_j)} = 1$, otherwise we say this edge is degenerate.
    Let $k_1 = \abs{\{i: E_i \text{ is simple}\}}$, let $k_2 = \abs{\{i: E_i \text{ is generic}\}}$ and let $k_3 = q_2 - k_1 - k_2$ be the number of degenerate edges.
    We are interested in counting configurations where $k_1 < q_2$.
    We will first show that non-simple covers with $k_3 = 0$ are unlikely to occur in $H(n,q_1,q_2,p)$. \par
    We say that a vertex is generic if it belongs to more than one generic edge and denote $m = \abs{\{v \in V(H'): v \text{ is generic} \}}$.
    Since every generic vertex belongs to at least two generic edges and every generic edge contains at most one generic vertex we obtain $1 \le m \le \frac{k_2}{2}$.
    There are $O(n^2)$ many possible choices for the short edge and $O(n^{k_1})$ many choices for the simple edges. Since $k_3 = 0$ every generic edge must contain a generic vertex, and thus, fixing all generic vertices determines the generic edges up to $O(1)$ many choices; thus, we have $O(n^m)$ many choices for the generic edges. \par
    We require $q_2$ vertices for the short edge, $q_1-1$ vertices for every simple edge and $k_2(q_1-1) - (k_2 - m)$ vertices for the generic edges. Hence, the expected number of such configurations is
    \begin{align*}
        O(n^{2 + k_1 + m} p^{q_2 + (k_1+k_2)(q_1-1) - (k_2 - m)})  &= O(n^{2 + k_1 + m}p^{q_2q_1 - (k_2 - m)})\\  
        = O(n^{2 + k_1 + m - (1 + \frac{1}{q_2-1})q_2 + \frac{(k_2 - m)q_2}{q_1(q_2-1)}}) &= O(n^{2 - (k_2 - m) - \frac{q_1q_2 - (k_2-m)q_2}{q_1(q_2-1)}}).
    \end{align*}
    If $k_2 - m \ge 2$, as $k_2 - m < q_1$ this is clearly $o(1)$.
    If $k_2 - m \le  1$ then since $k_2 \ge 2$ we must have $k_2 = 2$ and $m=1$ and therefore
    \begin{align*}
         O(n^{2 - (k_2 - m) - \frac{q_1q_2 - (k_2-m)q_2}{q_1(q_2-1)}}) = O(n^{1 - \frac{q_2(q_1-1)}{q_1(q_2-1)}}) = O(n^{1 - 1 - \frac{q_1 - q_2}{q_1(q_2-1)}}) = o(1).
    \end{align*}

    Therefore, we may assume that $k_3 > 0$.
    Assume first that there exists a degenerate edge $E_1$ along with two other covering edges $E_2, E_3$ such that $\abs{E_1 \cap E_i} = 1$ for $i \in \{2,3\}$ and $\abs{E_2 \cap E_3} < 2$. Then fixing $E_1$ determines all three edges up to $O(1)$ many choices. Thus, we have at most $O(n)$ many choices for those three edges.
    Since we require $q_2$ vertices for the short edge and $\abs{E_1 \cup E_2 \cup E_3} \ge 3q_1 - \abs{E_1 \cap E_2} - \abs{E_1 \cap E_3} - \abs{E_2 \cap E_3} > 3q_1 - 3$, the expected number of such configurations is
    \[
        O(n^3p^{3q_1 - 6 + q_2}) = O(n^3p^{3(q_1-1) + q_2- 3}) = O(n^{-3\cdot\frac{q_1 - q_2}{q_1(q_2-1)}}p^{q_2-3}) = o(1).
    \]

    We thus deduce that every degenerate edge must either intersect some other covering edge in at least two vertices or intersect two other degenerate edges that intersect each other in at least two vertices.
    Therefore, $k_3 \ge 2$, and by Lemma \ref{lemma:single-edge-multicover} we have at most $O(1)$ many choices for the degenerate edges. \par
    We now divide the generic edges into two categories: We say a generic edge is weakly-generic if it is a generic edge that lies in the same connected component of $H' \setminus \{E\}$ as a degenerate edge; otherwise, we say it is strongly-generic.
    We denote $k'_2 = \abs{\{i: E_i\text{ is strongly-generic}\}}$.
    We say that a vertex is strongly-generic if it belongs to more than one strongly-generic edge, and denote $m' = \abs{\{v \in V(H'): v \text{ is strongly-generic} \}}$.
    Since every strongly-generic vertex belongs to at least two strongly-generic edges and every strongly-generic edge contains at most one strongly-generic vertex we obtain $0 \le m' \le \frac{k'_2}{2} \le \frac{k_2}{2}$.
    Again, we note that fixing the strongly-generic vertices determines the strongly-generic edges up to $O(1)$ many choices.
    We also note that fixing $E$ determines the degenerate and weakly-generic edges up to $O(1)$ many choices.
    Since, by Lemma \ref{lemma:minimal-covers}, we require at least $2q_1(1-\frac{1}{k_3})$ vertices for the degenerate edges, at least $k_2(q_1 - 1)$ additional vertices for the generic edges and $k_1q_1$ vertices for the simple edges, we get that the expected number of such configurations is (since $k_1 + k_2 + k_3 = q_2$, $k_3 \ge 2$ and $m'\le \frac{k_2}{2}$)
    \begin{align*}
        &\ O(n^{2 + k_1 + m'}p^{k_1q_1 + k_2(q_1-1) + 2q_1(1-\frac{1}{k_3})}) \\
        &= O(n^{2 + k_1 + m' - k_1(1 + \frac{1}{q_2-1}) - k_2(1 + \frac{q_1 - q_2}{q_1(q_2-1)})}p^{2q_1(1-\frac{1}{k_3})})\\
        &= O(n^{2 - k_1\frac{1}{q_2-1} + m' - k_2 - k_2\frac{q_1-q_2}{q_1(q_2-1)}} p^{2q_1(1-\frac{1}{k_3})}),  \\
    \intertext{since $m' - k_2 \le -\frac{k_2}{2}$ and $p^{2q_1(1-\frac{1}{k_3})} = O(n^{-2 -\frac{2}{q_2-1} + \frac{2q_2}{k_3(q_2-1)}})$,} 
        &= O(n^{2-\frac{k_1}{q_2-1} - \frac{k_2}{2} -  k_2\frac{q_1-q_2}{q_1(q_2-1)} - 2 - \frac{2}{q_2-1} + \frac{2q_2}{k_3(q_2-1)}}) \\
        &= O(n^{-k_2\frac{q_1-q_2}{q_1(q_2-1)} -\frac{(2+k_1 +\frac{k_2}{2}(q_2-1))k_3 - 2q_2}{k_3(q_2-1)}}) \\
        &= O(n^{-k_2\frac{q_1-q_2}{q_1(q_2-1)} -\frac{2k_3+k_1k_3 +k_2k_3 - 2q_2}{k_3(q_2-1)}}) = o(1).
    \end{align*}
    Finally, we conclude by Markov's inequality,
    \[
        \mathbb{P}[\text{\#short edges with non-simple covers} > 0] = \sum_{k_1 < q_2} \sum_{k_2 < q_2 - k_1} o(1) = o(1).
    \]
\end{proof}

\begin{lemma}\label{lemma:simple-extension}
    Let $H \subset H(n, q_1, q_2)$ be a simple path of length $\ell = O(\log{n})$.
    Define $Y = \abs{\{T \subset H(n, q_1, q_2):  H \cup T \text{ is a simple path of length } \ell + 1\}}$.
Then for every $c_r > 0$, there exists $c>0$ such that for $p \le c \cdot n^{-\frac{q_2}{q_1(q_2-1)}}$ we have \[\mathbb{E}[Y \mid V(H) \subset [n]_p] < c_r.\] 
\end{lemma}
\begin{proof}
    Without loss of generality we may assume that the short edge in $T$ interesects $H$ only at $E_{\ell,q_2}$. We denote $T = \{E=E_{\ell+1}, E_{\ell+1,2}, \dotsc, E_{\ell+1,q_2}\}$.
Since $E$ intersects $H$ in a single vertex we have $O(n)$ many choices for it.
By fixing $E$ we limit the number of choices for each covering edge to $O(n)$, giving us a total of $O(n^{q_2})$ many choices for $T$ in $H(n,q_1,q_2)$. \\
Since $T$ requires $q_2 - 1 + (q_2-1)(q_1-1) = (q_2 - 1)q_1$ new vertices, we obtain
\[
    \mathbb{E}[Y \mid V(H) \subset [n]_p] = O(n^{q_2}p^{(q_2-1)q_1}) = O(n^{q_2 - q_2}c^{(q_2-1)q_1}) = O(c^{(q_2-1)q_1}).
\]
\end{proof}

\begin{corollary}\label{cor:simple-paths}
    Let $\ell = O(\log{n})$, and let $Y$ be the random variable counting simple paths of length $\ell$ in $H(n,q_1,q_2,p)$. Then,
    \[
        \mathbb{E}Y = O(c_r^\ell n^{1-\frac{1}{q_2-1}}).
    \]
\end{corollary}
\begin{proof}
    We may construct a simple path of length $\ell$ by first choosing a single long edge, and then choosing $\ell$ extensions.\par
    Since two integers belong to only a constant number of $q_1$-term arithmetic progressions, we obtain that the number of choices for the initial long edge is $O(n^2)$.
    Since we require $q_1$ vertices for the long edge, we obtain that the expected number of long edges is $O(n^2p^{q_1}) = O(n^{1-\frac{1}{q_2-1}})$. \par
    By the previous lemma, we may extend the path step by step, each time adding a factor of $c_r$ to the expectation, and the corollary immediately follows.
\end{proof}

We are now ready to prove the probabilistic lemma, thus completing the proof of Theorem~\ref{thm:asymmetric-0-statement}.

\begin{proof}[Proof of \nameref{lemma:asymmetric-probabilistic-lemma}]
Let $H=H(n,q_1,q_2,p)$, let $B$ be a sufficiently large constant and set $\ell' = B\log{n}$.
The proof of this lemma can be summarized as follows: \par
First, we recall that, by Lemma \ref{lemma:only-simple-covers}, a.a.s\  no non-simply covered short edges exist, and show that all simple paths terminate at lengths smaller than $\ell'$, using Corollary \ref{cor:simple-paths}.
We then apply a first-moment argument to several random variables, showing that every small $2$-blocking hypergraph contains some sub-hypergraph with $o(1)$ many expected copies in $H$; thus, the probability that the $2$-blocking hypergraph appears in our random hypergraph is $o(1)$. \par
Specifically, we will show that the existence of a $2$-blocking hypergraph implies either the existence of additional vertices such that the expected number of choices for them is $o(n^{-1 + \frac{1}{q_2-1}})$,
or the existence of a simple path that obeys an additional constraint that causes us to lose a degree of freedom in the path construction.
Since, by Corollary \ref{cor:simple-paths}, we have only an expected $O(n^{1-\frac{1}{q_2-1}})$ many choices for a simple path, we obtain that replacing a factor of $n$ with a factor of $O(\log^k{n})$ causes the expectation to tend to zero as $n$ grows. \par 
 As seen in the proof for Corollary \ref{cor:simple-paths}, we may construct a path by selecting a long edge and then iteratively extending the path; therefore, in some cases, we refer to a single long edge as a simple path of length zero, allowing us to treat a single block as an extension to an existing path. \par
We will also sometimes assume that the edges of a simple path are ordered $E_{1,1}, E_1, E_{1,2}, \dotsc, E_{1,q_2}, E_2, \dotsc, E_{\ell, q_2}$.
In such an ordering, if there are no other constraints on the path, we have $O(n^2)$ many choices for the first edge and $O(n)$ many choices for every other edge.
For convenience, we say a vertex $v \in V(P)$ precedes another vertex $u \in V(P)$ if $v$ belongs to an edge that precedes all edges that contain $u$. \par
Let $U$ be the random variable counting short edges with non-simple covers. By Lemma~\ref{lemma:only-simple-covers} we have
\[ \mathbb{P}(U > 0) = o(1). \]\par
Let $W$ be the random variable counting simple paths of length $\ell'$.
By Corollary \ref{cor:simple-paths} and Markov's inequality we obtain 
\[ \mathbb{P}(W > 0) \le \mathbb{E}W = O(n^{1-\frac{1}{q_2-1}}c_r^{B\log{n}}) = o(1). \]\par
Let $X$ be the random variable counting spoiled simple paths of length $\ell < \ell'$.
Let $P$ be a simple path and let $E$ be a spoiling edge for it.
By Corollary \ref{cor:simple-paths}, there are at most $O(n^{1-\frac{1}{q_2-1}})$ many choices for the path up to the final block.
For the final block we have $O(n)$ many choices for the short edge, and for all the long edges except $E_{\ell,q_2}$.
Since $\abs{E \cap (V(P)\setminus E_{\ell,q_2})} \ge 2$ we have at most $O(\log^2{n})$ many choices for $E$, and therefore $O(\log^2{n})$ many choices for $E_{\ell,q_2}$.
Since we require $q_2 - 1$ vertices for the short edge, along with $(q_2-1)(q_1 - 1)$ vertices for the long edges we conclude that
\[ \mathbb{E}X = \sum_{\ell < \ell'}O(n^{1-\frac{1}{q_2-1}}n^{q_2-1}p^{q_1(q_2-1)} \log^2{n}) = O(n^{q_2 - q_2 - \frac{1}{q_2-1}}\log^2{n}) = o(1), \]
and therefore by Markov's inequality,
\[ \mathbb{P}(X > 0) = o(1).\]\par

Let $Y$ be the random variable counting simple paths of length $\ell < \ell'$ with a saw. 
We denote the path by $P$.
Since the saw edges are not entirely contained in $P$, each one must contain a vertex $v \notin V(P)$.
Let $S_2, \dotsc, S_{q_1}$ be the saw edges. 
For each $k \in \{2,\dotsc,q_1\}$, let $\{s_k\} = (V(P) \cap S_k) \setminus E_{1,1}$ (there is only one such vertex, since $\abs{S_k \cap V(P)} = 2$), and let $S'_k$ be the first edge in $P$ that contains $s_k$. 
We split into several cases:
\medskip\\\medskip\noindent\textbf{Case 1.} There exist two edges $S_i, S_j$ such that $\abs{S_i \cap S_j} \ge 2$.\\ 
By Lemma \ref{lemma:single-edge-multicover}, we obtain that fixing $E_{1,1}$ determines $S_i$ and $S_j$ up to $O(1)$ many choices.
Thus, if we iteratively extend a path from $E_{1,1}$ we obtain that when we select $S'_i$ it must intersect one (or both) of $S_i,S_j$, giving us only $O(1)$ many choices for it, and thus the expected number of such paths with $S_i$ and $S_j$ in $H$ is $O(n^{-\frac{1}{q_2-1}}\log{n}) = o(1)$.
Therefore, the probability that such saw edges exist is $o(1)$ and we may assume all saw edges $S_i$ and $S_j$ intersect each other in at most one vertex.
\medskip\\\medskip\noindent\textbf{Case 2a.} There exist $i,j$ such that $(S_i \cap S_j) \setminus V(P) \ne \emptyset$ and $S'_i \ne S'_j$. \\
Without loss of generality we assume $S'_i$ follows $S'_j$ in the edge ordering of $P$.
If we fix the subpath up to (but not including) $S'_i$, then we have at most $O(\log{n})$ many choices for $S_j$ since it intersects both $E_{1,1}$ and $S'_j$ which have already been chosen;
this implies that there are further $O(1)$ many choices for $S_i$, as it must intersect both $S_j$ and $E_{1,1}$.
Therefore, there are further $O(1)$ many choices for $S'_i$ as it must intersect both $S_i$ and either a short edge or the final long edge of the previous block.
We now have $O(n)$ many choices for each remaining edge in the block that contains $S'_i$. \par
We obtain that for the final block, we have $O(n)$ many choices for each edge except for $S'_i$ for which we have only $O(\log{n})$ many choices.
Hence the expected number of such configurations is 
\[ 
\sum_{\ell<\ell'}O(n^{1-\frac{1}{q_2-1}}n^{q_2-1}p^{q_1(q_2-1)} \log{n}) = O(n^{q_2 - q_2 - \frac{1}{q_2-1}}\log^2{n}) = o(1).
 \]
\medskip\noindent\textbf{Case 2b.} There exist $i,j$ such that $(S_i \cap S_j) \setminus V(P) \ne \emptyset$ and $S'_i = S'_j = S'\ne E_1$. \\
By fixing $E_{1,1}$ and $\{v\} = S_i \cap S_j$ we determine $S_i$ and $S_j$ up to $O(1)$ many choices.
Since $v \notin V(P)$, we also determine $S'$ up to $O(1)$ many choices as it must intersect $S_i \cup S_j$ in two vertices other than $v$.
If we now extend a path from $E_{1,1}$ towards $S'$, we obtain that we have only $O(1)$ many choices for the edge connecting the path to $S'$ (whether it is a short edge, or the final long edge in a block). 
Note that this connecting edge cannot be $E_{1,1}$ by our assumption that $S' \ne E_1$.
Thus, the number of choices for both $S'$ and the previous edge is $O(n)$ - the number of choices for $v$. \par
If we now add the remaining edges in the final block, we obtain that the expected number of such paths is \[O(n^{-\frac{1}{q_2-1}}\log{n}) = o(1).\] 
\medskip\noindent\textbf{Case 2c.} There exist $i,j$ such that $(S_i \cap S_j) \setminus V(P) \ne \emptyset$ and $S'_i = S'_j = S' = E_1$. \\
From the previous cases we may assume that all saw edges intersect each other in at most one vertex, and any saw edge that intersects another saw edge outside of $V(P)$ must intersect $V(P)\setminus E_{1,1}$ in one vertex that lies on $E_1$.  \par

Let $A = \bigcup_{i=2}^{q_1}(S_i \setminus V(P))$ and for each $v \in A$ define $S(v) =\{i: v \in S_i\}$.
Suppose that  $\abs{S(v)} \ge 2$ for some $v \in A$.
Then, \[\abs{\bigcup_{i \in S(v)}(S_i \cap (E_1 \setminus E_{1,1}))} = \abs{S(v)},\] since no two saw edges share more than one vertex and any saw edge that intersects another has a vertex in $E_1\setminus E_{1,1}$.
Therefore, $\abs{S(v)} \le \abs{E_1 \setminus E_{1,1}} = q_2-1$ for all $v \in A$.
Since,
\[(q_2-2)(q_1-1) = \sum_{i=2}^{q_1}\abs{S_i \setminus V(P)} = \sum_{v \in A}\abs{S(v)} \le \abs{A}(q_2-1),\] we deduce that we require $\abs{A} \ge  \frac{(q_2 - 2)(q_1 - 1)}{q_2 - 1} = q_1 - 1 - \frac{q_1-1}{q_2-1}$ additional vertices for the saw edges. \par
If we assume $P$ is fixed, we have at most $O(\log{n})$ many choices for each saw edge. Thus, the expected number of choices for the additional vertices is 
\begin{align*}
    O(\log^{q_1-1}{n}\cdot p^{q_1 -1 - \frac{q_1-1}{q_2-1}}) &=
    O(\log^{q_1-1}{n}\cdot n^{-1 - \frac{1}{q_2-1} + \frac{q_2}{q_1(q_2-1)} + \frac{q_1q_2}{q_1(q_2-1)^2} - \frac{q_2}{q_1(q_2-1)^2}}) \\
    &= O(\log^{q_1-1}{n}\cdot n^{-1 - \frac{q_1-q_2}{q_1(q_2-1)} + \frac{1}{q_2-1} +  \frac{q_1 - q_2}{q_1(q_2-1)^2}}) \\
    &= O(\log^{q_1-1}{n}\cdot n^{-1 + \frac{1}{q_2-1} - \frac{(q_1-q_2)(q_2-2)}{q_1(q_2-1)^2}}).
\end{align*} \par
Thus by Corollary \ref{cor:simple-paths}, the expected number of choices for the path and the saw edges is \[O(\log^{q_1}{n} \cdot n^{-\frac{(q_1-q_2)(q_2-2)}{q_2(q_2-1)^2}}) = o(1).\]
\medskip\noindent\textbf{Case 3.} $(S_i \cap S_j) \setminus V(P) = \emptyset$ for all $i \ne j \in \{2,\dotsc,q_1\}$. \\
Since each saw edge contains at least one vertex not in $V(P)$, there are at least $q_1-1$ additional vertices introduced by the saw edges.
We have $O(\log{n})$ many choices for each saw edge as they must intersect $V(P)$ in two vertices each, one of which lies on $E_{1,1}$.
Thus, the expected number of paths with saws such as above is
\begin{align*} 
    \sum_{\ell<\ell'}O(n^{1-\frac{1}{q_2-1}}(\log{n}\cdot p)^{q_1-1}) &=
O(\log^{q_1}{n}\cdot n^{1 - \frac{1}{q_2-1} - \frac{q_2}{q_2-1} + \frac{q_2}{q_1(q_2-1)}}) \\
                                                                   &= O(\log^{q_1}{n} \cdot n^{-\frac{2q_1-q_2}{q_1(q_2-1)}}) = o(1).
\end{align*}
Thus, by Markov's inequality,
\[
    \mathbb{P}(Y > 0) = o(1).
\]\par
Finally, let $Z$ be the random variable counting simple paths of length $\ell <\ell'$ with a spoiled extension.
Once more, we divide the argument into several cases.
First, let $Z_1$ be the random variable counting simple paths with spoiled extensions such that no long edge in the extension intersects the path $P$ in more than one vertex. \par
We denote the number of long edges in the extension that intersect $V(P)$ by $k$.
We have an expected $O(n^{1-\frac{1}{q_2-1}})$ many choices for $P$, $O(n)$ many choices for the short edge and each of the $q_2-1-k$ long edges that are disjoint from $V(P)$, and $O(\log{n})$ many choices for each of the $k$ long edges that intersect $P$.
We also require $k(q_1-1) + (q_2-1-k)q_1$ new vertices. 
Thus, the expected number of such configurations is
\begin{align*}
    \mathbb{E}Z_1 &= \sum_{\ell<\ell'}\sum_{k=1}^{q_2}O(n^{1-\frac{1}{q_2-1} + q_2 - k}p^{(q_2-1-k)q_1 + k(q_1-1)}\log^k{n}) \\ 
                  &= \sum_{k=1}^{q_2}O(n^{1 - \frac{1}{q_2-1} + q_2 - k}p^{(q_2-1)q_1 - k} \log^{q_2+1}{n}) \\
                  &= \sum_{k=1}^{q_2}O(n^{1 - \frac{1}{q_2-1} - k(1 - \frac{q_2}{q_1(q_2-1)})} \log^{q_2+1}{n}) \\
                  &= O(n^{- \frac{q_1-q_2}{q_1(q_2-1)}} \log^{q_2+1}{n}) = o(1).
\end{align*} \par
Next, let $Z_2$ be the random variable counting paths with spoiled extensions that contain exactly one long edge $L$ that intersects the path $P$ in at least two vertices.
Fixing $P$, we have only $O(\log^2{n})$ many choices for $L$ and $O(1)$ further choices for $E_{\ell+1}$, as it must intersect both $L$ and $E_{\ell,q_2}$.
Once again, we sum over $k$, the number of long edges in the extension that intersect $V(P)$ in exactly one vertex.
We have $O(n^{q_2-2-k})$ choices for all edges of the extension,
and we require at least $(q_2-2)q_1 - k + 1$ new vertices.
Thus, since $k \le q_2-2$ and 
\[
    p^{q_1(q_2-2)} = O(n^{-\frac{q_2(q_2-2)}{q_2-1}}) = O(n^{-q_2 +1 + \frac{1}{q_2-1}}),
\]
the expected number of such configurations is 
\begin{align*}
    \mathbb{E}Z_2 &= \sum_{\ell<\ell'}\sum_{k=0}^{q_2-2}O(n^{q_2-k-1-\frac{1}{q_2-1}}p^{(q_2-2)q_1 - k + 1}\log^2{n})\\
                  &= \sum_{k=0}^{q_2-2}O(n^{-k}p^{1-k}\log^3{n}) 
                  = O(\log^3{n} \cdot p) = o(1).
\end{align*} \par
Finally, let $Z_3$ be the random variable counting paths with spoiled extensions such that there exist distinct $i,j \in [q_2]$ such that $\abs{V(P) \cap E_{\ell+1,z}} \ge 2$ for $z \in \{i,j\}$.
We denote $L_z = E_{\ell+1,z}$. \par
Since $E_{\ell+1}$ must have a simple cover we obtain that $L_z \cap E_{\ell,q_2} = \emptyset$ for $z \in \{i,j\}$. 
Therefore, if we fix the path $P$ up to (but not including) $E_{\ell,q_2}$, we have at most $O(\log^2{n})$ many choices for each of $L_i$ and $L_j$, since both edges intersect $V(P)$ in at least two vertices.
Therefore, we have at most $O(\log^4{n})$ many choices for $E_{\ell+1}$ as it must intersect both $L_i$ and $L_j$. \par
Finally, we obtain that we have $O(\log^4{n})$ many choices for $E_{\ell,q_2}$ as it must intersect both $E_{\ell+1}$ and $E_\ell$.
Since we have an expected $O(n^{1-\frac{1}{q_2-2}})$ many choices for the path up to the final block, and we have $O(n^{q_2-1}\log^4{n})$ many choices for the final block (but still require $q_1(q_2-1)$ vertices), the expected number of choices for $P$ is
\[
    \sum_{\ell<\ell'}O(n^{-\frac{1}{q_2-1}}\log^4{n}) = o(1),
\]
hence
\[
    \mathbb{P}(Z_3 > 0) = o(1),
\]
and therefore
\[
    \mathbb{P}(Z > 0) \le \mathbb{P}(Z_1 + Z_2 + Z_3 > 0) = o(1).
\]
Finally, \[\mathbb{P}(U=W=X=Y=Z=0) \to 1 \text{ as } n \to \infty\] completing the proof.
\end{proof}

%% file: acknowledgments.tex
I would like to thank my advisor Wojciech Samotij for his introduction to this problem and for his guidance towards its solution. I would also like to express my deepest appreciation for his kindness, patience and meticulousness during the research and writing process.